\documentclass[a4paper,11pt]{article}
\usepackage[margin=1in]{geometry}
\usepackage{amsmath,amssymb,amsthm,mathrsfs,amsfonts,bm}
\usepackage{mathtools}
\usepackage{graphicx}
\usepackage[numbers]{natbib}
\usepackage{setspace}
\usepackage{subfigure}
\usepackage{enumerate}
\usepackage{hyperref}
\usepackage{multirow}
\usepackage{diagbox}
\usepackage{nonfloat}
\newcommand{\R}{\mathbb{R}}

\providecommand{\abs}[1]{\lvert#1\rvert}

\DeclareMathOperator*{\argmin}{argmin}
\newtheorem{theorem}{Theorem}[section]
\newtheorem{lemma}{Lemma}[section]
\newtheorem{proposition}{Proposition}[section]
\newtheorem{corollary}{Corollary}

\newtheorem{remark}{Remark}[section]
\theoremstyle{definition}
\numberwithin{equation}{section}

\newtheorem{assumption}{Assumption}

\begin{document}

\title{\textbf{Optimal Drift Rate Control and Impulse Control for a Stochastic Inventory/Production System}}

\author{Ping Cao\thanks{School of Management, University of Science and Technology of China, Hefei, 230026, China.
    pcao@ustc.edu.cn}\ \ and
Dacheng Yao\thanks{Academy of Mathematics and Systems Science, Chinese Academy of Sciences, Beijing, 100190, China. dachengyao@amss.ac.cn}}

\date{}
\maketitle

\begin{abstract}
In this paper, we consider joint drift rate control and impulse control for a stochastic inventory system under long-run average cost criterion.
Assuming the inventory level must be nonnegative, we prove that a $\{(0,q^{\star},Q^{\star},S^{\star}),\{\mu^{\star}(x): x\in[0, S^{\star}]\}\}$ policy is an optimal joint control policy,
where the impulse control follows the control band policy $(0,q^{\star},Q^{\star},S^{\star})$,
that brings the inventory level up to $q^{\star}$ once it drops to $0$ and brings it down to $Q^{\star}$ once it rises to $S^{\star}$,
and the drift rate only depends on the current inventory level and is given by function $\mu^{\star}(x)$ for the inventory level $x\in[0,S^{\star}]$.
The optimality of the $\{(0,q^{\star},Q^{\star},S^{\star}),\{\mu^{\star}(x): x\in[0,S^{\star}]\}\}$  policy is proven by using a lower bound approach, in which a critical step is to prove the existence and uniqueness of optimal policy parameters. To prove the existence and uniqueness, we develop a novel analytical method to solve a free boundary problem consisting of an ordinary differential equation (ODE) and several free boundary conditions.
Furthermore, we find that the optimal drift rate $\mu^{\star}(x)$ is firstly increasing and then decreasing as $x$ increases from $0$ to $S^{\star}$ with a turnover point between $Q^{\star}$ and $S^{\star}$.
\end{abstract}

\textbf{Keywords:}
drift rate control; impulse control; Brownian motion; inventory control

\textbf{2010 Mathematics Subject Classification:}
90B05, 93E20, 60J70,49N25, 49K15

\section{Introduction}

In this paper, we study a continuous-review stochastic inventory/production system, in which supply/production rate and inventory level can be adjusted.
The \emph{netput} inventory level process, capturing the difference of regular supply/production process
and the demand process, has the following representation
\begin{equation}
\label{eq:W}
W_t=x_0+\int_0^t \mu_s \, \mathrm{d}s+\sigma B_t,\quad t\geq0,
\end{equation}
where $W_0=x_0\in\mathbb{R}^+$ is the initial inventory level, $\mu_t$ is the drift rate at time $t$ and is a decision variable,
$\sigma^2>0$ is the variance, and $B=\{B_t:t\geq0\}$ is a standard one-dimensional Brownian motion with $B_0=0$.
The system manager can modify the drift rate at any time,
and can also increase and decrease the inventory level at any time by any amount desired.
Let $\bm{\mu}=\{\mu_t\in\mathbb{R}: t\geq0\}$ denote the drift rate control process,
and let $\bm{Y}=(Y_1,Y_2)$ be a pair of impulse controls with $Y_i=\{(\tau_{n}^i,\xi_{n}^i):n\geq 1\}$, $i=1,2$,
where $\tau_n^1$ ($\tau_n^2$) represents the $n$th time to increase (decrease) the inventory level and
$\xi_n^1$ ($\xi_n^2$) denotes the corresponding increment (decrement).
Then, the controlled inventory level is given by
\begin{equation}
\label{eq:X-1}
X_t=x_0+\int_0^t \mu_s \, \mathrm{d}s+\sigma B_t+\sum_{n=1}^{N_t^1} \xi_n^1-\sum_{n=1}^{N_t^2} \xi_n^2,\quad t\geq0,
\end{equation}
where $N_t^i=\sup\{n:\tau^i_n\leq t\}$ denotes the number of adjustments of $Y_i$ up to time $t$. Moreover, the inventory level must be nonnegative at all times.
There are three types of costs: the holding cost and drift rate cost are continuously incurred and depend on
 current inventory level and drift rate respectively, and the impulse control cost is incurred when the inventory is increased/decreased and depends on the increment/decrement.
The objective is to find a control policy $(\bm{\mu},\bm{Y})$ to minimize the average expected total costs over an infinite planning horizon.

The joint drift rate control and impulse control model described above has many applications of practical interest.
The following are two examples.
\begin{enumerate}
\item[(i)] \textit{Joint pricing and inventory control problems}.
For example, let $-\mu_t$ denote the demand rate at time $t$ and depend on the current product price $p_t$, i.e.,
\[
\mu_t=\mu(p_t),
\]
and let $\sigma^2$ denote the demand variance. Thus, the cumulative demand up to time $t$ is given by $-\int_0^t \mu(p_s)\,\mathrm{d}s-\sigma B_t$.
In such problems, the manager controls the drift rate by adjusting the offered price, while he controls the inventory level by ordering and returning products.
The joint pricing and inventory control problems with Brownian motion demand have been studied in the literature, e.g., \cite{ChenWuYao2010,Yao2017,ZhangZhang2012}. However, only upward impulse control (i.e., ordering)  is
allowed in these works.

\item[(ii)] \textit{Production-inventory problems}.
For example, let $\delta t-\sigma B_t$ denote the cumulative customer demand up to time $t$.
The system produces products at rate $\lambda_t$ for time $t$. Besides this standard production,
it also can place an expedited order to an outside supplier when the inventory level becomes too low and can return some products when the level becomes too high.
In this kind of production-inventory problem, the drift rate $\mu_t$ in \eqref{eq:X-1} becomes
\[
\mu_t=\lambda_t-\delta.
\]
Production-inventory problems also have been considered in the literature; see e.g., \cite{Bradley2004,WuChao2014} for dual production rate models.
\end{enumerate}

Because of their importance in practice, both drift rate control and impulse control have been widely studied in the literature.
Next, we first briefly review the literature on drift rate control, then review the literature on two-sided impulse control.
Finally, we introduce the our work with joint drift rate control and impulse control.
See Table \ref{tab:literature} for the classification of the related literature.

\begin{table}[h]
\centering
\resizebox{\textwidth}{!}{
\begin{tabular}{|c|c|c|c|c|}
\hline
\multicolumn{2}{|c|}{\multirow{2}{*}{}} &  \multicolumn{3}{|c|}{Impulse control} \\
\cline{3-5}
\multicolumn{2}{|c|}{} & None &  One-sided & Two-sided  \\
\hline
\multirow{4}{*}{Drift rate control} & None &  NA & --- & \cite{Constantinides1976,ConstantinidesRichard1978,DaiYao2013a,DaiYao2013b,HarrisonSellkeTaksar1983,OrmeciDaiVandeVate2008,Richard1977}\\
\cline{2-5}
 & Two modes & \cite{AvramKaraesmen1996,Doshi1978,Rath1977,WuChao2014} &  --- & None\\
\cline{2-5}
 & Finite modes & \cite{ChernoffPetkau1978,OrmeciVandeVate2011,OrmeciVandeVateWang2015,SongYinZhu2012}&  --- & None \\
\cline{2-5}
 & Any value & \cite{AtaHarrisonShepp2005,GhoshWeerasinghe2007}  &  \cite{ChenWuYao2010,Yao2017,ZhangZhang2012}& our paper \\
\hline
\end{tabular}
}
\caption{Literature review for drift rate control and impulse control\protect\footnotemark}
\label{tab:literature}
\end{table}
\footnotetext{In this table,  ``---" denotes that the references are not related directly to our work and thus are not included.}

In the literature on drift rate control, two-mode models with positive switching costs have been widely studied;
see e.g., \cite{AvramKaraesmen1996,Doshi1978,Rath1977,WuChao2014}.
These four papers prove the optimality of an $(m,M)$ policy under different cost criteria.
Under an $(m,M)$ policy, the system uses the lower drift rate mode once the system's state reaches or exceeds $M$, uses the higher drift rate mode once the system's state drops to or below $m$,
and keeps the drift rate mode unchanged otherwise.
Furthermore, finite drift rate modes (i.e., more than two modes) are considered in \cite{ChernoffPetkau1978,OrmeciVandeVate2011,OrmeciVandeVateWang2015,SongYinZhu2012},
where simple Brownian motion models are considered in \cite{ChernoffPetkau1978,OrmeciVandeVate2011,OrmeciVandeVateWang2015}
while a general diffusion process model is studied in \cite{SongYinZhu2012}.
Note that all the works introduced above only consider the drift rate control in their models.
Recently, models in which the drift rate can take any value have been studied in \cite{AtaHarrisonShepp2005,GhoshWeerasinghe2007};
Ghosh and Weerasinghe \cite{GhoshWeerasinghe2007} study a joint drift rate control and singular control problem and explicitly solve it under a quadratic control cost structure.

Two-sided impulse control problems with constant drift rate also have been widely studied in the literature:
Long-run average cost models are studied in \cite{Constantinides1976,DaiYao2013a,OrmeciDaiVandeVate2008}
and discounted cost models are studied in \cite{ConstantinidesRichard1978,DaiYao2013b,HarrisonSellkeTaksar1983,Richard1977}.
These works all prove the optimality of a control band policy $(d,D,U,u)$,
under which the state is immediately increased to level $D$ once it drops to $d$
and is decreased to level $U$ once it goes up to $u$.
The method for proving the existence of optimal policy parameters in these works
is to obtain an explicit solution for a relative optimality equation represented by an ODE,
and then to find parameters to satisfy some free boundary conditions derived from the optimal impulse control.
However, this method cannot work for our problem, since it is difficult to obtain the explicit solution of the corresponding ODE due to the changeable drift rate.
To overcome this difficulty, we prove the existence of an optimal control policy by analyzing the ODE with
the associated free boundary conditions directly.

There are also some papers considering joint drift rate control and impulse control, all of which focus on the application to joint pricing and inventory control problems (\cite{ChenWuYao2010,Yao2017,ZhangZhang2012}).
There are two important differences between those papers and this one.
First, their models only allow increasing inventory in the impulse control, while we allow both increasing and decreasing inventory.
Second and more important, Chen et al. \cite{ChenWuYao2010} and Zhang and Zhang \cite{ZhangZhang2012} prove the existence of optimal parameters only when the price is constant in a certain inventory interval and when $\mu(\cdot)$ has a specific form, respectively.
Although Yao \cite{Yao2017} considers a general drift rate function like ours,
he assumes the existence of optimal policy parameters for the optimality equation.
We, however, completely prove the existence of optimal policy parameters by solving a free boundary problem, which is the main technical contribution in this paper.

This paper's contribution can be summarized as follows. First,
to the best of the authors' knowledge, this is the first study of a stochastic inventory problem with joint drift rate control and two-sided impulse control. Further, the optimal control policy is completely characterized.
Second, a novel method is provided to prove the existence and uniqueness of optimal policy parameters by solving a free boundary problem. This is the major technical contribution of this paper.
Indeed, proving the existence of optimal policy parameters for Brownian control problems is usually important and difficult, and has been the major technical contribution of many publications; see e.g., \cite{BenkheroufBensoussan2009,BensoussanLiuSethi2005,DaiYao2013a,DaiYao2013b,FengMuthuraman2010}.
However, in contrast with these works, this paper develops a very different method that does not require an explicit solution for the optimality equation. This methodology provides a more general roadmap to solve similar problems, especially when the explicit expression of the solution is unavailable or too complicated.
Third, unlike the simple monotonic optimal drift rate studied in the literature (see e.g., \cite{AtaHarrisonShepp2005,GhoshWeerasinghe2007}),
we find that the optimal drift rate function $\mu^{\star}(x)$ as a function of inventory level $x$,
is first increasing and then decreasing as the inventory level $x$ increases in $[0,S^{\star}]$.

The rest of this paper is organized as follows.
In \S \ref{sec:model},
we introduce the mathematical formulation of the joint drift rate control and impulse control problem in \S \ref{sec:model-sub}, and state our main results in \S \ref{sec:mainresults-sub}.
In \S \ref{sec:existence}, a policy is provided by proving the existence of its parameters, and
\S \ref{sec:optimality} proves the optimality of proposed policy.
Finally, the paper concludes in \S \ref{sec:concluding}. We close this section with some frequently used notation. Let $x^+=\max\{0,x\}$, $\mathbb{R}^+=[0,\infty)$, and $\mathscr{C}^1(\mathbb{R}^+)$
be the space of continuous functions on $\mathbb{R}^+$ that have continuous first derivatives.
Let $f$ be a real-valued function defined on $\mathbb{R}^+$, and use $f_{t-}$ to denote the left limit at time point $t$.

\section{Formulation and main results}
\label{sec:model}
\subsection{Model formulation}
\label{sec:model-sub}

Let $(\Omega,\{\mathcal{F}_t\},\mathcal{F},\mathbb{P})$ be a filtered probability space
and Brownian motion $B=\{B_t:t\geq0\}$ is adapted with respect to the filtration $\{\mathcal{F}_t\}$.

Consider an inventory system with a \emph{netput} inventory level process given by \eqref{eq:W}.
There are two controls for this system: a drift rate control $\bm{\mu}=\{\mu_t:t\geq0\}$ and a two-sided impulse control $\bm{Y}=(Y_1,Y_2)$ with $Y_i=\{(\tau_{n}^i,\xi_{n}^i):n\geq 1\}$, $i=1,2$. These two controls together form
a policy $\phi=(\bm{\mu},\bm{Y})$.
A joint drift rate control and impulse control policy $\phi=(\bm{\mu},\bm{Y})$ is \emph{admissible} if: i) $\mu_t$ is $\mathcal{F}_t$-measurable and $\mu_t$ must be in a compact set $\mathcal{U}$ with the smallest element $\underline{\mu}$ and the largest element $\bar{\mu}$; and ii)  $\tau_n^i$ is a stopping time and $\xi_n^i$ is $\mathcal{F}_{\tau_n^i-}$-measurable. We note that $\mathcal{U}$ might be a discrete point set or an interval and $\underline{\mu}$ and $\bar{\mu}$ are both finite.
Let $\Phi$ denote the set of all admissible policies.
Under an admissible policy $\phi$, the controlled inventory level process $X$ must be nonnegative and is given by \eqref{eq:X-1}.

We next introduce three costs in our system.
The holding cost is continuously charged at rate $h(x)$ when the inventory level is $x$
and the drift rate control cost is continuously charged at rate $c(\mu)$ when the rate is $\mu$.
Furthermore, the impulse control cost is measured by the amount of adjustment, and cost $K+k\xi$ is incurred when quantity $\xi$ is increased while cost $L+\ell \xi$
is incurred when quantity $\xi$ is decreased, where $K$, $k$, $L$, and $\ell$ are all strictly positive. Therefore, under an admissible policy $\phi=(\bm{\mu},\bm{Y})$, the system's long-run average cost is
\begin{align*}
\mathcal{C}(x_0,\phi)
=\limsup_{t\to\infty}\frac{1}{t} \mathbb{E}_{x_0}\Big[\int_0^t \big(h(X_s)+c(\mu_s)\big)\, \mathrm{d}s
+\sum_{n=1}^{N_t^1} \big(K+k\xi_n^1\big)+\sum_{n=1}^{N_t^2} \big(L+\ell\xi_n^2\big)\Big],
\end{align*}
where $X_{0-}=x_0\in\mathbb{R}^+$ is the initial inventory level and $\mathbb{E}_{x_0}[\cdot]$ denotes the expectation with respect to the initial inventory level $x_0$.
Our objective is to find an admissible policy $\phi^{\star}=(\bm{\mu}^{\star},\bm{Y}^{\star})$
such that for any $x_0\in\mathbb{R}^+$,
\begin{equation}
\label{eq:problem}
\mathcal{C}(x_0,\phi^{\star})=\inf_{\phi\in\Phi}\mathcal{C}(x_0,\phi).
\end{equation}


To this end, we use the following assumptions about the holding cost function $h$.
\begin{assumption}
\label{assum-h}
$h(x)$ is strictly increasing and continuous in $x$ with $x\geq 0$ and $h(0)=0$. Moreover, $\lim_{x\to\infty}h(x)=\infty$.
\end{assumption}

The assumptions for the function $h$ are quite standard and are satisfied in most applications of practical interest; see e.g., a linear cost function $h(x)=hx$ in \cite{HarrisonSellkeTaksar1983,OrmeciDaiVandeVate2008} and a convex function in \cite{DaiYao2013a}.
Our assumptions are used to ensure the existence and uniqueness of optimal policy parameters, i.e., Theorem \ref{thm:existence} and Corollary \ref{cor:unique}; see e.g., the proof of Lemma \ref{lem:w-existence}-\ref{lem:w-property}.
It is worth noting that the ``strictly increasing'' may be relaxed to ``weakly increasing''
without jeopardizing the existence but losing the uniqueness of  optimal policy parameters.
Of course, this relaxation would require a tedious analysis.
Also, the condition $h(0)=0$ can be relaxed to $h(0)=a$ for some $a\geq0$ without jeopardizing the main results.

Notice that no condition is imposed on the drift rate function $c$.
In this paper, we will see that the function $c$ is used only in Lemma \ref{lem:property_pi_mu} and the proof of Lemma~\ref{lem:gamma2} \eqref{item-b-gamma2},
which require no conditions on $c(\mu)$ since the allowable drift rate set $\mathcal{U}$ is assumed to be a compact set.
However, if $\mathcal{U}$ is no longer a compact set, e.g., $\mathcal{U}=(-\infty,\infty)$, then we should impose some regular conditions on $c$
such as differentiability, convexity, $c'(-\infty)=-\infty$, and $c'(\infty)=\infty$, to guarantee the correctness of the main results.  Moreover, our analysis relies on the finiteness of $\underline{\mu}$ and $\bar{\mu}$. If either of them is infinite, some analysis must be changed accordingly and a more lengthy analysis is required.

\begin{remark}
Notice that in our model, we assume that the drift rate can be controlled but the variance is a constant.
There are two reasons for this assumption: First, because of the technical difficulties for the continuous adjustment when the variance is affected by the control,  it is a common assumption in the literature; see e.g.,
\cite{AtaHarrisonShepp2005,AvramKaraesmen1996,GhoshWeerasinghe2007,OrmeciVandeVate2011,OrmeciVandeVateWang2015}. Second, this assumption is reasonable for many situations in practice.
For example, the uncertainty in demand of branded products (e.g., Intel processors) that exhibit substantial customer demand is mainly due to a random error that is independent of the decision variable (\cite{LiZheng2006}).
\end{remark}

\subsection{Main results}
\label{sec:mainresults-sub}

In this subsection, the main results are presented. In Theorem \ref{thm:existence}, we determine a $\phi^{\star}=\{(0,q^{\star},Q^{\star},S^{\star}), \{\mu^{\star}(x):x\in[0,S^{\star}]\}\}$ policy by solving an ODE with some boundary conditions. Then, in Theorem \ref{thm:mainresults}, we show that
this policy is optimal for the joint drift rate control and impulse control problem \eqref{eq:problem}.

For a $\phi=\{(0,q,Q,S), \{\mu(x):x\in[0,S]\}\}$ policy, $(0,q,Q,S)$ with $0<q<Q<S$ denotes a two-sided impulse control policy and $\{\mu(x):x\in[0,S]\}$ denotes a drift rate control policy.
Under an impulse control policy $(0,q,Q,S)$, the inventory level is increased up to level $q$ instantaneously once it drops to level 0 and is decreased down to $Q$ instantaneously once it rises to level $S$.
Thus, $Y_1$ can be specified as
\begin{align*}
 \tau_n^1=
 \begin{cases}
 \inf\{t\geq0: X_{t-}= 0\} & \text{if $n=1$,}\\
 \inf\{t>\tau_{n-1}^1:X_{t-}= 0\} & \text{if $n\geq 2$,}
 \end{cases}
 \quad \text{and}
 \quad
 \xi_n^1=q \text{ for $n\geq1$.}
\end{align*}
Since the initial inventory level $x_0$ may be higher than $S$, a return with amount $x_0-Q$ may happen at time 0.
Therefore, $Y_2$ can be specified as
\begin{align*}
 \tau_n^2=
 \begin{cases}
 \inf\{t\geq0: X_{t-}\geq S\} & \text{if $n=1$,}\\
 \inf\{t>\tau_{n-1}^2:X_{t-}= S\} & \text{if $n\geq 2,$}
 \end{cases}
 \text{ and }
 \xi_n^2=
 \begin{cases}
 \max\{S, x_0\}-Q& \text{if $n=1$,}\\
 S-Q  & \text{if $n\geq2$}.
 \end{cases}
\end{align*}
The impulse control policy $(0,q,Q,S)$ is called a \emph{control band policy} in the literature
(see e.g., \cite{HarrisonSellkeTaksar1983,OrmeciDaiVandeVate2008}).
Note that under control band policy $(0,q,Q,S)$, the controlled inventory level $X_t$ is limited to $[0,S]$ for all $t\geq0$.
Under a drift rate control policy $\{\mu(x):x\in[0,S]\}$, the drift rate would be $\mu(x)$ when the inventory level is $X_{t}=x\in[0,S]$ at any time $t\geq0$.

In developing the following theorem, we will determine a $\phi=\{(0,q,Q,S), \{\mu(x):x\in[0,S]\}\}$ policy.
To present the theorem, we first define some functions as follows.
Let $w$ be a real-valued function on $\mathbb{R}$, and define
\begin{equation}
\label{eq:pi}
\pi(w)=\min_{\mu\in\mathcal{U}}\big(\mu w+c(\mu)\big)\quad
\text{and}\quad \mu(w)=\argmin_{\mu\in\mathcal{U}}\big(\mu w+c(\mu)\big),
\end{equation}
where we choose $\mu(w)$ to be the smallest one if there are multiple minimizers.
\begin{theorem}
\label{thm:existence}
Assume that $h(\cdot)$ satisfies Assumption \ref{assum-h}.
\begin{enumerate}[$(a)$]
\item
\label{item-thm-existence-a}
There exist four parameters $q^{\star}$, $Q^{\star}$, $S^{\star}$ and $\gamma^{\star}$ with
$0<q^{\star}<Q^{\star}<S^{\star}$ and $\gamma^{\star}\in\mathbb{R}$, and a continuously differentiable function
$w^{\star}(\cdot):\mathbb{R}^+\to\mathbb{R}$ satisfying
\begin{align}
\label{eq:ODE}
\frac{1}{2}\sigma^2 \frac{ \mathrm{d} w^{\star}(x)}{ \mathrm{d} x}
+\pi(w^{\star}(x))+h(x)=\gamma^{\star}, \quad \text{for $x\in[0,S^{\star}]$}
\end{align}
with boundary conditions
\begin{align}
&\int_0^{q^{\star}}\big[w^{\star}(x)+k\big]\, \mathrm{d}x=-K,\label{eq:int=-K}\\
&\int_{Q^{\star}}^{S^{\star}}\big[w^{\star}(x)-\ell\big]\, \mathrm{d}x=L.\label{eq:int=L}\\
&w^{\star}(q^{\star})=-k,\label{eq:w(q)=-k}\\
&w^{\star}(Q^{\star})=w(S^{\star})=\ell.\label{eq:w(Q)=w(S)=l}
\end{align}
\item
\label{item-thm-existence-b}
Define $\mu^{\star}(x)\triangleq\mu(w^{\star}(x))$. Then $\phi^{\star}=\{(0,q^{\star},Q^{\star},S^{\star}), \{\mu^{\star}(x):x\in[0,S^{\star}]\}\}$ is an admissible policy. Furthermore, there exists a number $x^{\star}$ with $x^{\star}\in(Q^{\star},S^{\star})$ such that
$\mu^{\star}(x)$ is decreasing in $x\in [0,x^{\star}]$ and increasing in $x\in [x^{\star},\infty)$.
\end{enumerate}
\end{theorem}

In this paper, \eqref{eq:int=-K}-\eqref{eq:w(Q)=w(S)=l} are called free boundary conditions since the boundary points $q^{\star}$, $Q^{\star}$, and $S^{\star}$ need to be determined, and problem \eqref{eq:ODE} with conditions \eqref{eq:int=-K}-\eqref{eq:w(Q)=w(S)=l} is also called free boundary problem; see a similar definition in \cite{DaiYao2013a}.
The optimality of the selected policy $\phi^{\star}=\{(0,q^{\star},Q^{\star},S^{\star}), \{\mu^{\star}(x):x\in[0,S^{\star}]\}\}$ is shown in the following theorem.

\begin{theorem}
\label{thm:mainresults}
Assume that $h(\cdot)$ satisfies Assumption \ref{assum-h}.
Then, the policy $\phi^{\star}=\{(0,q^{\star},Q^{\star},S^{\star}), \{\mu^{\star}(x):x\in[0,S^{\star}]\}\}$ with parameters defined in
Theorem \ref{thm:existence} is an optimal policy among all admissible policies, and $\gamma^{\star}$ is the optimal long-run average cost.
\end{theorem}

Theorem \ref{thm:mainresults} has shown that $\gamma^{\star}$ is the optimal cost and thus it is unique in
Theorem \ref{thm:existence}. This in turn implies the uniqueness of optimal policy parameters $q^{\star}$, $Q^{\star}$, and $S^{\star}$ and the function $w^{\star}(\cdot)$ in Theorem \ref{thm:existence}.
\begin{corollary}
\label{cor:unique}
The optimal impulse control parameters  $q^{\star}$, $Q^{\star}$, and $S^{\star}$  and the function $w^{\star}(\cdot)$  in Theorem \ref{thm:existence} are unique.
\end{corollary}

Finally, we give a heuristic derivation of the ODE \eqref{eq:ODE} and free boundary conditions \eqref{eq:int=-K}-\eqref{eq:w(Q)=w(S)=l} that the optimal parameters should satisfy.
For a given policy $\phi=\{(0,q,Q,S),\{\mu(x): x\in[0,S]\}\}$,
let $V(x)$ be the difference of the expected cumulative cost from state $x\in\mathbb{R}^+$ to state $0$
and the cost $\gamma \tau(x,0)$, where $\gamma$ denotes the average cost under policy $\phi$ and $\tau(x,0)$ is the first time
when the inventory level hits $0$ starting from $x$. In the literature, $V$ is also called the relative value function; see e.g., \cite{OrmeciDaiVandeVate2008}. Let
\[
w^{\star}(x)=V'(x).
\]
First, the definition of $V$ implies that $V$ should satisfy $V(0)=V(q)+K+kq$ and $V(S)=V(Q)+L+\ell (S-Q)$,
which yield \eqref{eq:int=-K} and \eqref{eq:int=L}. Next, we show that $V$ should satisfy \eqref{eq:ODE}, \eqref{eq:w(q)=-k}, and \eqref{eq:w(Q)=w(S)=l}
if $\phi$ is optimal.
If $\phi$ is optimal, by the principle of optimality,
for $X_0=x\in(0,S)$ and a small time interval with length $\delta$,
$V(x)$ should satisfy
\[
V(x)=\min_{\mu_s\in\mathcal{U},s\in[0,\delta]}\mathbb{E}_{x}\Big[\int_0^{\delta}\big( h(X_s)+c(\mu_s)\big)\,\mathrm{d}s-\gamma \delta+V(X_{\delta})|X_0=x,\mu_0=\mu\Big]+o(\delta)
\]
with $X_{s}=x+\int_0^s \mu_\upsilon \,\mathrm{d}\upsilon+\sigma B_{s}$ for $s\in[0, \delta]$.
It follows from a standard argument for the dynamic programming equation (see e.g., \cite{FlemingSoner2006}) that $V(x)$ satisfies
\[
\frac{1}{2}\sigma^2 V''(x)+\min_{\mu\in\mathcal{U}}\big(\mu V'(x)+c(\mu)\big)+h(x)-\gamma=0,
\]
which implies \eqref{eq:ODE}.
Furthermore, starting from state $S$, if it is optimal to jump state $Q$,
then $Q$ should be chosen by minimizing $V(Q)+L+\ell(S-Q)$.
The first-order optimality condition would be $V'(Q)=\ell$, which is the first equality in \eqref{eq:w(Q)=w(S)=l}.
Besides this, for $x\geq S$, under the policy $\phi$, we must have $V(x)=V(Q)+L+\ell(x-Q)$.
By the principle of smoothness under the optimal policy, the left and right derivatives of $V$ at $S$
should be equal, i.e., $V'(S)=\ell$; that is the second equality in \eqref{eq:w(Q)=w(S)=l}.
Finally, a similar analysis gives us \eqref{eq:w(q)=-k}.

We will prove Theorem \ref{thm:existence} in \S\ref{sec:existence} and
Theorem \ref{thm:mainresults} and Corollary \ref{cor:unique} in \S\ref{sec:optimality}.

\section{Existence of optimal policy parameters}
\label{sec:existence}

In this section, we prove Theorem~\ref{thm:existence},
which shows the existence of a policy (which by Theorem \ref{thm:mainresults} is optimal) with parameters satisfying (\ref{eq:ODE})-(\ref{eq:w(Q)=w(S)=l}).
Specifically, we prove Theorem~\ref{thm:existence} in two subsections.
In \S\ref{sec:ode}, we solve the ODE \eqref{eq:ODE} with given $\gamma\in\mathbb{R}$
and a boundary condition $w(0)=w_0\in\mathbb{R}$, and provide the structural properties of $w$
with respect to $x$, $w_0$, and $\gamma$; see Lemmas \ref{lem:w-existence}-\ref{lem:w-property}.
Then, in \S\ref{sec:optimal_parameters}, we determine
$(w_0^{\star},\gamma^{\star},q^{\star},Q^{\star},S^{\star})$ by the five boundary conditions
\eqref{eq:int=-K}-\eqref{eq:w(Q)=w(S)=l}; see Lemmas \ref{lem:gamma1}-\ref{lem:gamma1=gamma2}.
Note that Lemmas \ref{lem:w-existence}-\ref{lem:gamma1=gamma2} are under Assumption \ref{assum-h}, although for brevity we will not state this for each.

Before proving Theorem~\ref{thm:existence}, recalling the definitions of $\pi(w)$ and $\mu(w)$ in \eqref{eq:pi}, we first give the following lemma, the proof of which can be found in Appendix \ref{app:proof-lemma-1}.

\begin{lemma}
\label{lem:property_pi_mu}
$\pi(w)$ is concave and Lipschitz continuous in $w\in\mathbb{R}$, i.e.,
for any $w_1$ and $w_2$, we have
\begin{equation}
\label{equ:Lipschitz_pi}
\abs{\pi(w_1)-\pi(w_2)}\leq M\abs{w_1-w_2},
\end{equation}
where $M=\max\{\abs{\underline{\mu}},\abs{\bar{\mu}}\}$.
Furthermore, $\mu(w)$ is decreasing in $w$.
\end{lemma}

\subsection{Solving the ODE \eqref{eq:ODE}}
\label{sec:ode}

In this subsection, we will solve the ODE \eqref{eq:ODE} for $x\geq0$ by assigning an initial value $w(0)=w_0\in\mathbb{R}$ and fixing $\gamma\in\mathbb{R}$, which then allows us to characterize the structural and asymptotical properties of the solution $w(\cdot)$.

Consider the following problem:
\begin{align}
&\frac{1}{2}\sigma^2 w'(x)+\pi(w(x))+h(x)=\gamma \quad\text{for $x\geq 0$}, \label{equ:ode_new}\\
&\quad \text{subject to } w(0)=w_0.  \nonumber
\end{align}
We denote the solution of the above problem by  $w(x;w_0,\gamma)$ if it exists.
For this problem, we first have the following lemma, which states the existence, uniqueness, and continuity of $w(x;w_0,\gamma)$.
\begin{lemma}
\label{lem:w-existence}
\begin{enumerate}[$(a)$]
\item
\label{item-a-lemma-w}
For any $w_0\in \mathbb{R}$ and $\gamma\in \mathbb{R}$,
problem \eqref{equ:ode_new} has a unique continuously differentiable solution $w(x; w_0,\gamma)$.
\item
\label{item-b-lemma-w}
$w(x; w_0,\gamma)$ is continuous in $w_0\in\mathbb{R}$ and $\gamma\in\mathbb{R}$, and
$w'(x; w_0,\gamma)$ is continuous in $x\in\mathbb{R}^+$, $w_0\in\mathbb{R}$, and $\gamma\in\mathbb{R}$ respectively.
\end{enumerate}
\end{lemma}
\begin{proof}
($a$) Since $\pi(w)$ is Lipschitz continuous (see Lemma~\ref{lem:property_pi_mu}) and $h(x)$ is continuous (see Assumption \ref{assum-h}),
using an analog to the proof for Proposition 3 (i) in \cite{AtaTongarlak2013}, we can use Picard's existence theorem
(see, e.g., Theorem 10 of \S1.7 in \cite{AdkinsDavidson_2012}) to show that there exists a unique continuous solution $w(x; w_0,\gamma)$ to (\ref{equ:ode_new}) on the interval $[0, \infty)$.

($b$) It follows from Theorem II-1-2 in \cite{HsiehSibuya_1999} that part \eqref{item-a-lemma-w} and the continuity of $h(x)$ imply that
$w(x; w_0,\gamma)$ is continuous in $w_0\in\mathbb{R}$ and $\gamma\in\mathbb{R}$.
Further, (\ref{equ:ode_new})  and the continuity of $h$, $\pi$, and $w$
immediately imply that $w'(x; w_0,\gamma)$ is continuous in $x\in\mathbb{R}^+$, $w_0\in\mathbb{R}$, and $\gamma\in\mathbb{R}$ repectively.
\end{proof}

It follows from (\ref{equ:Lipschitz_pi}) (by letting $w_1=w(x; w_0,\gamma)$ and $w_2=0$) and (\ref{equ:ode_new}) that
\begin{align}
&\frac{1}{2}\sigma^2w'(x; w_0,\gamma)+M|w(x; w_0,\gamma)|+\pi(0)+h(x)\geq\gamma,\quad\text{and}\label{equ:inequality_w_x}\\
&\frac{1}{2}\sigma^2w'(x; w_0,\gamma)-M|w(x; w_0,\gamma)|+\pi(0)+h(x)\leq\gamma. \label{equ:inequality_w_x_another}
\end{align}
which will be used in the following discussions.

The following lemma characterizes the monotonicity and asymptotical behaviors of $w(x; w_0,\gamma)$ with respect to $\gamma\in\mathbb{R}$ and $w_0\in\mathbb{R}$.
\begin{lemma}
\label{lem:w-property-1}
The following results hold:
\begin{enumerate}[$(a)$]
\item
\label{item-c-lemma-w}
For fixed $x>0$ and $w_0\in\mathbb{R}$,
$w(x; w_0,\gamma)$ is strictly increasing in $\gamma\in\mathbb{R}$
and
\begin{equation}
\label{eq:lim-w=infty-gamma}
\lim_{\gamma\to\pm\infty}w(x; w_0,\gamma)=\pm\infty.
\end{equation}
\item
\label{item-d-lemma-w}
For fixed $x\geq0$ and $\gamma\in\mathbb{R}$, $w(x;w_0,\gamma)$ is strictly increasing in $w_0\in\mathbb{R}$
and
\[
\lim_{w_0\to\pm\infty}w(x; w_0,\gamma)=\pm\infty.
\]
\end{enumerate}
\end{lemma}

\begin{proof}

($a$) First, if $\gamma_1<\gamma_2$, we show that $w(x; w_0, \gamma_1)<w(x; w_0,\gamma_2)$ for any fixed $x>0$ and $w_0\in\mathbb{R}$.
Fix $w_0\in \mathbb{R}$. Suppose that $w(x;w_0,\gamma_1)\geq w(x;w_0,\gamma_2)$ for some $x>0$.
Define
\[
f_{\gamma}(x)=w(x; w_0,\gamma_2)-w(x; w_0,\gamma_1)\quad\text{and}\quad
x_{\gamma}=\inf\{x>0: f_{\gamma}(x)\leq 0\}.
\]
The following proof is divided into two cases: $x_{\gamma}>0$ or $x_{\gamma}=0$.

If $x_{\gamma}>0$, then it follows from the continuity of $w(\cdot; w_0,\gamma_i)$, $i=1,2$, that
$f_{\gamma}(x_{\gamma})=0=f_{\gamma}(0)$ and $f_{\gamma}(x)>0$ for all $x\in (0, x_{\gamma})$.
By the continuity of $f_{\gamma}(\cdot)$, there exist two numbers $x_1, x_2\in(0, x_{\gamma})$ with $x_1<x_2$ such that
\begin{equation}
\label{eq:f>f}
f_{\gamma}(x_1)>f_{\gamma}(x_2)\quad \text{and}\quad
M f_{\gamma}(x)<\gamma_2-\gamma_1 \text{ for all $x\in[x_1,x_2]$,}
\end{equation}
where $M=\max\{\abs{\underline{\mu}},\abs{\bar{\mu}}\}$ (see Lemma \ref{lem:property_pi_mu}).
Moreover, it follows from (\ref{equ:ode_new}) that
\begin{equation}\label{equ:ode_gamma}
\frac{1}{2}\sigma^2f_{\gamma}'(x)+\pi(w(x; w_0,\gamma_2))-\pi(w(x; w_0,\gamma_1))=\gamma_2-\gamma_1, \text{ for }x\geq 0.
\end{equation}
Integrating (\ref{equ:ode_gamma}) from $x_1$ to $x_2$, we have
\begin{align*}
&(\gamma_2-\gamma_1)\cdot(x_2-x_1)\\
&\quad=\frac{1}{2}\sigma^2(f_{\gamma}(x_2)-f_{\gamma}(x_1))+\int_{x_1}^{x_2}\big[\pi(w(y; w_0, \gamma_2))-\pi(w(y; w_0, \gamma_1))\big]\, \mathrm{d}y\\
&\quad< \int_{x_1}^{x_2}\big[\pi(w(y; w_0, \gamma_2))-\pi(w(y; w_0, \gamma_1))\big]\, \mathrm{d}y\\
&\quad\leq  \int_{x_1}^{x_2}Mf_{\gamma}(y)\, \mathrm{d}y\\
&\quad\leq (\gamma_2-\gamma_1)\cdot(x_2-x_1),
\end{align*}
where the first and last inequalities follow from \eqref{eq:f>f} and the second inequality follows from \eqref{equ:Lipschitz_pi}.
This is a contradiction.

If $x_{\gamma}=0$, then there exists a sequence $\{\hat{x}_n, n\in\mathbb{N}\}$ such that $\hat{x}_n\downarrow 0$ as $n\to\infty$ and $f_{\gamma}(\hat{x}_n)\leq 0$. Therefore,
$$\frac{f_{\gamma}(\hat{x}_n)}{\hat{x}_n}\leq 0, \text{ for all }n\geq 1.$$
Since $f_{\gamma}(0)=0$, taking the limit as $n\to\infty$ gives $f_{\gamma}'(0)\leq 0$, which contradicts (\ref{equ:ode_gamma}) with $x=0$.
Therefore, the contradictions for both $x_{\gamma}>0$ and $x_{\gamma}=0$ imply that $w(x; w_0,\gamma)$ is strictly increasing in $\gamma\in\mathbb{R}$ for any given $x>0$ and $w_0\in\mathbb{R}$.

Next, we prove that $\lim_{\gamma\to\infty}w(x; w_0,\gamma)=\infty$ for any given $x>0$ and $w_0\in\mathbb{R}$.
Fix $w_0\in\mathbb{R}$. There must exist a number $\gamma_3$ such that for all $\gamma\geq\gamma_3$, both
\begin{equation}\label{gamma_constraint_1}
\gamma>\pi(0)+h(x)
\end{equation}
and
\begin{equation}\label{gamma_constraint_2}
w_0e^{\xi y}+\frac{2}{\sigma^2}\int_0^y\big[\gamma-\pi(0)-h(z)\big]e^{\xi(y-z)}\,\mathrm{d}z\geq0
\text{ for all $y\in[0,x]$}
\end{equation}
hold. We claim that for any fixed $x_3\in(0,x)$, it holds that
\begin{equation}
\label{eq:w>0}
 w(y; w_0,\gamma)\geq 0\quad \text{for all $y\in [x_3,x]$ and $\gamma\geq\gamma_3$}.
\end{equation}
First, we show that for each $\gamma\geq \gamma_3$, there exists a number $x_4\in(0,x_3]$ (depending on $\gamma$ and $w_0$) such that
$w(x_4;w_0,\gamma)\geq0$. If not, there exists a number $\gamma_4$ with $\gamma_4\geq \gamma_3$ such that $w(y; w_0,\gamma_4)<0$ for all $y\in(0,x_3]$, and then \eqref{equ:inequality_w_x} implies that for all $y\in(0,x_3]$,
\[
\frac{1}{2}\sigma^2w'(y; w_0,\gamma_4)-Mw(y; w_0,\gamma_4)\geq \gamma_4-\pi(0)-h(y),
\]
which implies that for all $y\in(0,x_3]$,
\begin{eqnarray*}
w(y;w_0,\gamma_4)&\geq&w_0e^{\xi y}
+\frac{2}{\sigma^2}\int_0^y\big[\gamma_4-\pi(0)-h(z)\big]e^{\xi(y-z)}\,\mathrm{d}z\geq0,
\end{eqnarray*}
where $\xi=2M/\sigma^2$ and the last inequality follows from \eqref{gamma_constraint_2}, $\gamma_4\geq \gamma_3$, and $y\leq x_3<x$.
This contradiction implies that for any given $\gamma\geq\gamma_3$, there exists a number $x_4\in(0,x_3]$ such that
$w(x_4;w_0,\gamma)\geq0$.
It follows from \eqref{equ:ode_new} and \eqref{gamma_constraint_1} that for any $y\in[0,x]$ and $\gamma\geq\gamma_3$, if $w(y; w_0,\gamma)=0$, we have $w'(y; w_0,\gamma)>0$.
Thus, the continuity of $w(\cdot; w_0,\gamma)$ and $w(x_4; w_0,\gamma)\geq 0$ imply
\eqref{eq:w>0}. Hence, \eqref{equ:inequality_w_x} implies that for all $y\in[x_3,x]$ and $\gamma\geq\gamma_3$,
\[
\frac{1}{2}\sigma^2w'(y; w_0,\gamma)+Mw(y; w_0,\gamma)\geq \gamma-\pi(0)-h(y),
\]
and thus
\begin{eqnarray*}
w(x; w_0,\gamma)&\geq&w(x_3; w_0,\gamma)e^{-\xi (x-x_3)}+\frac{2}{\sigma^2}\int_{x_3}^{x}\big[\gamma-\pi(0)-h(y)\big]e^{-\xi(x-y)}\,\mathrm{d}y\\
&\geq & \frac{2}{\sigma^2}\int_{x_3}^{x}\big[\gamma-\pi(0)-h(y)\big]e^{-\xi(x-y)}\,\mathrm{d}y
\end{eqnarray*}
for all $\gamma\geq\gamma_3$. Letting $\gamma\to\infty$ in the inequality above yields $\lim_{\gamma\to\infty}w(x; w_0,\gamma)=\infty$.

Similar to the proof of $\lim_{\gamma\to\infty}w(x; w_0,\gamma)=\infty$, except that \eqref{equ:inequality_w_x_another} is used instead of \eqref{equ:inequality_w_x}, one can show that $\lim_{\gamma\to-\infty}w(x; w_0,\gamma)=-\infty$. The detailed proof is omitted for brevity.

($b$) Choose any two numbers $w_0^{\dag}$ and $w_0^{\ddag}$ satisfying $w_0^{\dag}<w_0^{\ddag}$.
We want to show that $w(x; w_0^{\dag}, \gamma)<w(x; w_0^{\ddag},\gamma)$ for all $x\geq 0$ and $\gamma\in\mathbb{R}$.
Fix $\gamma\in\mathbb{R}$.
Obviously, the condition holds for $x=0$. Suppose that $w(x;w_0^{\dag},\gamma)\geq w(x;w_0^{\ddag},\gamma)$ for some $x>0$. Define
\[
f_w(x)=w(x; w_0^{\ddag},\gamma)-w(x; w_0^{\dag},\gamma)\quad \text{and}\quad
x_w=\inf\{x>0: f_w(x)\leq 0\}.
\]
Since $w(0;w_0^{\dag})=w_0^{\dag}<w_0^{\ddag}=w(0;w_0^{\ddag})$,
it follows from the continuity of $w(\cdot; w_0^{\dag},\gamma)$ and $w(\cdot; w_0^{\ddag},\gamma)$ that $x_w>0$, $f_w(x_w)=0$
and $f_w(x)>0$ for all $x\in[0,x_w)$.
By the continuity of $f_w(\cdot)$, there exists a number $x_5\in(0, x_w)$ such that
\begin{equation}
\label{eq:x-5}
M(x_w-x_5)<\sigma^2/2\quad \text{and}\quad f_w(x)\leq f_w(x_5)\quad \text{for all $x\in[x_5, x_w]$}.
\end{equation}
Moreover, it follows from \eqref{equ:ode_new} that
\begin{equation}
\label{equ:ode_w0}
\frac{1}{2}\sigma^2f_w'(x)+\pi(w(x; w_0^{\ddag},\gamma))-\pi(w(x; w_0^{\dag},\gamma))=0\quad \text{for $x\geq 0$}.
\end{equation}
Integrating (\ref{equ:ode_w0}) from $x_5$ to $x_w$, we have
\begin{align*}
0&=\frac{1}{2}\sigma^2(f_w(x_w)-f_w(x_5))+\int_{x_5}^{x_w}\big[\pi(w(y; w_0^{\ddag}, \gamma))-\pi(w(y; w_0^{\dag}, \gamma))\big]\,\mathrm{d}y\\
&<-\frac{1}{2}\sigma^2f_w(x_5)+\int_{x_5}^{x_w}Mf_w(y)\,\mathrm{d}y
\\
&\leq -\frac{1}{2}\sigma^2f_w(x_5)+(x_w-x_5)Mf_w(x_5)
\\
&<0,
\end{align*}
where the first inequality follows from $f_w(x_w)=0$ and \eqref{equ:Lipschitz_pi},
and the last two inequalities follow from \eqref{eq:x-5}.
This contradiction implies that $w(x; w_0,\gamma)$ is strictly increasing in $w_0$ for all $x\geq 0$ and $\gamma\in\mathbb{R}$.

Finally, the proof of $\lim_{w_0\to\pm\infty}w(x; w_0,\gamma)=\pm\infty$ is quite similar to that of $\lim_{\gamma\to\pm\infty}w(x; w_0,\gamma)=\pm\infty$ and is omitted for brevity.
\end{proof}

The following lemma characterizes the monotonic properties of $w(x; w_0,\gamma)$ with respect to $x\in\mathbb{R}^+$ when
$\gamma$ takes different values.

\begin{lemma}
\label{lem:w-property}
Fix $w_0\in\R$. There exists an upper bound $\bar{\gamma}(w_0)$ (possibly infinite) with $\bar{\gamma}(w_0)>\pi(w_0)$ such that the following holds.
\begin{enumerate}[$(a)$]
\item\label{solution_x_1}
If $\gamma\leq\pi(w_0)$,
then $w(x;w_0,\gamma)$ is strictly decreasing in $x\in(0,\infty)$ and
\begin{equation}
\label{eq:lim=-infty}
\lim_{x\to\infty}w(x; w_0,\gamma)=-\infty.
\end{equation}
\item\label{solution_x_2} If $\gamma\geq\bar{\gamma}(w_0)$,
then $w(x; w_0,\gamma)$ is strictly increasing in $x\in[0,\infty)$ and
\begin{equation}
\label{eq:lim=infty}
\lim_{x\to\infty}w(x;w_0,\gamma)=\infty.
\end{equation}
\item\label{solution_x_3}
If $\pi(w_0)<\gamma<\bar{\gamma}(w_0)$, then there exists a unique number $x^{\star}(w_0,\gamma)$
such that $w(x; w_0,\gamma)$ is strictly increasing in $x\in[0,x^{\star}(w_0,\gamma)]$
and strictly decreasing in $x\in[x^{\star}(w_0,\gamma),\infty)$.
Furthermore, $\lim_{x\to\infty}w(x; w_0,\gamma)=-\infty$.
\end{enumerate}
\end{lemma}
\begin{remark}
\label{rem:x-star}
For convenience, in the following sections, we let $x^{\star}(w_0,\gamma)=0$ if $\gamma\leq \pi(w_0)$ and $x^{\star}(w_0,\gamma)=\infty$ if $\gamma\geq\bar{\gamma}(w_0)$.
\end{remark}
\begin{proof}
Before proving \eqref{solution_x_1}-\eqref{solution_x_3}, we first claim the following  property:
There do not exist two numbers $x_1$ and $x_2$ with $x_1<x_2$ such that
\begin{equation}
\label{eq:w=w,w'<w'}
w(x_1; w_0,\gamma)=w(x_2;w_0,\gamma)\quad \text{and}\quad
w'(x_1; w_0,\gamma)\leq 0\leq w'(x_2;w_0, \gamma).
\end{equation}
This is because: it follows from (\ref{equ:ode_new}) with $x=x_1$ and $x=x_2$ that
\[
\frac{\sigma^2}{2} w'(x_1; w_0,\gamma)+\pi(w(x_1; w_0,\gamma))+h(x_1)=\frac{\sigma^2}{2}w'(x_2; w_0,\gamma)+\pi(w(x_2; w_0,\gamma))+h(x_2),
\]
which is impossible considering \eqref{eq:w=w,w'<w'} and $h(x_1)<h(x_2)$ (using Assumption~\ref{assum-h}).
Therefore, we have
\begin{enumerate}[(i)]
\item\label{item-i-property} $w(x; w_0,\gamma)$ can-not have a local minimizer in $x\in(0,\infty)$; and
\item\label{item-ii-property} $w(x; w_0,\gamma)$ can-not be a constant in any interval $[x_1,x_2]$ with $0\leq x_1<x_2<\infty$.
\end{enumerate}

Next we use properties \eqref{item-i-property}-\eqref{item-ii-property} to prove \eqref{solution_x_1}-\eqref{solution_x_3}.
First, taking $x=0$ in \eqref{equ:ode_new} and noting $h(0)=0$, we have
\begin{equation}
\label{eq:w'(0)}
w'(0; w_0,\gamma)=\frac{2}{\sigma^2}\cdot(\gamma-\pi(w_0)).
\end{equation}

($a$) The proof of monotonicity is divided into two cases: $\gamma<\pi(w_0)$ or $\gamma=\pi(w_0)$.

If $\gamma<\pi(w_0)$, \eqref{eq:w'(0)} implies that $w'(0;w_0,\gamma)<0$.
The continuity of $w'(x; w_0,\gamma)$ with respect to $x$ and properties \eqref{item-i-property}-\eqref{item-ii-property}
immediately imply that $w(x;w_0,\gamma)$ is strictly decreasing in $x$ for $x>0$.

If $\gamma=\pi(w_0)$, using the results when $\gamma<\pi(w_0)$ and the continuity of $w'(x; w_0,\gamma)$ in $\gamma$ ,
we must have $w'(x;w_0,\gamma)\leq 0$ for $x>0$, and property \eqref{item-ii-property}
further implies that $w(x;w_0,\gamma)$ is strictly decreasing in $x$ for $x>0$.

Next we show that $\lim_{x\to\infty}w(x; w_0,\gamma)=-\infty$.
Otherwise, there exists a finite number $\underline{w}$ such that
$\lim_{x\to\infty}w(x; w_0,\gamma)=\underline{w}$ and thus $\lim_{x\to\infty}w'(x; w_0,\gamma)=0$.
Taking $x\to\infty$ in (\ref{equ:ode_new}) yields that $\lim_{x\to\infty}h(x)=\gamma-\pi(\underline{w})$,
which contradicts with Assumption~\ref{assum-h}.

($b$) Define
\begin{equation}
\label{eq:bar-gamma}
\bar{\gamma}(w_0)=\sup\{\gamma\in\mathbb{R}: \mbox{there exisits an $x>0$ such that $w'(x; w_0,\gamma)<0$}\}.
\end{equation}
Part \eqref{solution_x_1} implies that $\bar{\gamma}(w_0)\geq \pi(w_0)$ and thus $\bar{\gamma}(w_0)$ is well defined although it might be $\infty$.
If $\gamma\geq\bar{\gamma}(w_0)$, then by the definition of $\bar{\gamma}(w_0)$ in \eqref{eq:bar-gamma}
and the continuity of $w'(x; w_0,\gamma)$ in $\gamma$,
we have that $w'(x; w_0,\gamma)\geq 0$ for all $x\geq 0$.
Then property \eqref{item-ii-property} implies that $w(x; w_0,\gamma)$ is strictly increasing in $x$.

The proof of  $\lim_{x\to\infty}w(x; w_0,\gamma)=\infty$ is similar to that of \eqref{eq:lim=-infty} and thus is omitted.

($c$) First, we must have
\[
\bar{\gamma}(w_0)>\pi(w_0).
\]
Otherwise, there is a contradiction between part \eqref{solution_x_1} and \eqref{solution_x_2} when $\gamma=\pi(w_0)=\bar{\gamma}(w_0)$. Now, we consider the case when $\pi(w_0)<\gamma<\bar{\gamma}(w_0)$.
First, we claim that for each $\gamma\in(\pi(w_0),\bar{\gamma}(w_0))$,
there exists a number $x>0$ such that $w'(x; w_0,\gamma)<0$.
If no such $x$ exists, we have that $w'(x;w_0,\gamma)\geq 0$ for all $x>0$.
Using arguments similar to those used to prove \eqref{eq:lim=-infty} in part \eqref{solution_x_1},
we can obtain
\begin{equation}
\label{eq:lim=infty-1}
\lim_{x\to\infty}w(x; w_0,\gamma)=\infty.
\end{equation}
On the other hand,
using the definition of $\bar{\gamma}(w_0)$ and the continuity of $w(x;w_0,\gamma)$ in $\gamma$,
there exists a number $\gamma^{\dag}\in(\gamma, \bar{\gamma}(w_0))$ such that $w'(x^{\dag}; w_0,\gamma^{\dag})<0$ for some $x^{\dag}>0$.
Then, we must have that $w(x; w_0,\gamma^{\dag})$ is strictly decreasing for $x\geq x^{\dag}$ and
\begin{equation}
\label{eq:lim=-infty-1}
\lim_{x\to\infty}w(x;w_0,\gamma^{\dag})=-\infty
\end{equation}
(the analysis is very similar to that of part \eqref{solution_x_1} and thus is omitted).
However, \eqref{eq:lim=infty-1} and \eqref{eq:lim=-infty-1} contradict
Lemma~\ref{lem:w-property-1}~$\eqref{item-c-lemma-w}$ with $\gamma^{\dag}>\gamma$.
Thus, we have proven that for each $\gamma\in(\pi(w_0),\bar{\gamma}(w_0))$,
there exists a number $x>0$ such that $w'(x; w_0,\gamma)<0$.
Define
\[
x^{\star}(w_0,\gamma)=\inf\{x\geq 0: w'(x; w_0,\gamma)<0\}.
\]
Since $w'(0; w_0,\gamma)=\gamma-\pi(w_0)>0$ and $w'(x; w_0, \gamma)$ is continuous in $x$, we have $x^{\star}(w_0,\gamma)>0$,
$w'(x; w_0,\gamma)\geq 0$ for $x\in[0, x^{\star}(w_0,\gamma))$ and $w'(x^{\star}(w_0,\gamma); w_0,\gamma)=0$.
Further, the properties \eqref{item-i-property}-\eqref{item-ii-property} imply that
$w(x;w_0,\gamma)$ is strictly increasing in $x\in[0,x^{\star}(w_0,\gamma)]$
and strictly decreasing in $x\in[x^{\star}(w_0,\gamma),\infty)$.

Finally, the proof of $\lim_{x\to\infty}w(x; w_0,\gamma)=-\infty$ is very similar to that of \eqref{eq:lim=-infty} and thus is omitted.
\end{proof}

\subsection{Determining optimal parameters by \eqref{eq:int=-K}-\eqref{eq:w(Q)=w(S)=l}}
\label{sec:optimal_parameters}
In the previous section, we obtained structural and asymptotical properties of solution $w(x;w_0,\gamma)$ to \eqref{eq:ODE}. In this subsection, we use these properties to find the optimal policy parameters $(q^{\star},Q^{\star},S^{\star})$
and auxiliary parameters $(w_0^{\star},\gamma^{\star})$
such that the boundary conditions \eqref{eq:int=-K}-\eqref{eq:w(Q)=w(S)=l} are satisfied.

Specifically, in Lemma \ref{lem:gamma1}, we show that for any $w_0<\ell$,
there exist unique $\gamma_1^{\star}(w_0)$, $Q(w_0)$, and $S(w_0)$
with $\gamma_1^{\star}(w_0)\in(\pi(w_0),\bar{\gamma}(w_0))$ and $0<Q(w_0)<x^{\star}(w_0,\gamma_1^{\star}(w_0))<S(w_0)$
such that
\begin{align}
&w(Q(w_0); w_0,\gamma_1^{\star}(w_0))=w(S(w_0); w_0,\gamma_1^{\star}(w_0))=\ell\quad\text{and}\label{eq:w=w=l}\\
&\int_{Q(w_0)}^{S(w_0)}\big[w(x;w_0,\gamma_1^{\star}(w_0))-\ell\big]\,\mathrm{d}x=L.\label{eq:int=L-gamma-1}
\end{align}
In Lemma \ref{lem:gamma2}, we prove that for any $w_0<\bar{w}_0<-k$,
there exist unique $\gamma_2^{\star}(w_0)$ and $q(w_0)$ with $\gamma_2^{\star}(w_0)\in(\pi(w_0),\infty)$
and $0<q(w_0)<x^{\star}(w_0,\gamma_2^{\star}(w_0))$ such that
\begin{align}
w(q(w_0);w_0,\gamma_2^{\star}(w_0))=-k\quad\text{and}\quad
\int_{0}^{q(w_0)}\big[w(x;w_0,\gamma_2^{\star}(w_0))+k\big]\,\mathrm{d}x=-K,\label{eq:int=-K-gamma-2}
\end{align}
where $\bar{w}_0$ is a number satisfying $\bar{w}_0<-k$.
Finally, in Lemma \ref{lem:gamma1=gamma2},
we show that we can choose a number $w_0^{\star}$ with $w_0^{\star}<\bar{w}_0$ such that
\begin{equation}
\label{eq:gamma1=gamma2}
\gamma_1^{\star}(w_0^{\star})=\gamma_2^{\star}(w_0^{\star}).
\end{equation}
Let $\gamma^{\star}=\gamma_1^{\star}(w_0^{\star})$, $q^{\star}=q(w_0^{\star})$,
$Q^{\star}=Q(w_0^{\star})$, $S^{\star}=S(w_0^{\star})$, $x^{\star}=x^{\star}(w_0^{\star}, \gamma^{\star})$,
and $w^{\star}(x)=w(x;w_0^{\star},\gamma^{\star})$.
Figure \ref{fig:optimal_function} depicts the function $w^{\star}$ and the optimal policy parameters.
Using Lemmas \ref{lem:gamma1}-\ref{lem:gamma1=gamma2}, we can prove Theorem \ref{thm:existence} as follows.
\begin{proof}[Proof of Theorem \ref{thm:existence}]
Recall that $w^{\star}$ is a continuously differentiable solution to \eqref{eq:ODE},
so \eqref{eq:w=w=l}-\eqref{eq:int=-K-gamma-2} with \eqref{eq:gamma1=gamma2} ensure \eqref{eq:int=-K}-\eqref{eq:w(Q)=w(S)=l}.
Besides, since $\gamma^{\star}=\gamma_1^{\star}(w_0^{\star})\in(\pi(w_0),\bar{\gamma}(w_0))$, Lemma \ref{lem:w-property}
\eqref{solution_x_3} and  $w^{\star}(q^{\star})=-k<\ell=w^{\star}(Q^{\star})$ imply $q^{\star}<Q^{\star}$,
and then $0<q^{\star}<Q^{\star}<S^{\star}$.
Thus we have finished the proof of Theorem \ref{thm:existence} \eqref{item-thm-existence-a}.
Furthermore, the facts that $\mu(\cdot)$ is decreasing (see Lemma \ref{lem:property_pi_mu}) and $w^{\star}(x)$ is increasing in $x\in [0,x^{\star}]$ and decreasing in $x\in [x^{\star},\infty)$ imply that
$\mu^{\star}(x)=\mu(w^{\star}(x))$ is decreasing in $x\in [0,x^{\star}]$ and increasing in $x\in [x^{\star},\infty)$.
This completes the proof of Theorem \ref{thm:existence} \eqref{item-thm-existence-b}.
\end{proof}

\begin{remark}
\label{rem:uniqueness}
It can be seen from the proof of Theorem \ref{thm:existence} that if we can prove the uniqueness of $w_0^{\star}$ in \eqref{eq:gamma1=gamma2}, then the uniqueness of parameters $q^{\star}$, $Q^{\star}$,
$S^{\star}$ and $\gamma^{\star}$ can be immediately obtained. However, it is very difficult to prove the uniqueness of $w_0^{\star}$ in \eqref{eq:gamma1=gamma2} directly. Instead, we will first show the uniqueness of $\gamma^{\star}$ using Theorem \ref{thm:mainresults}, then the uniqueness of $w_0^{\star}$ follows by noting that $\gamma_1^{\star}(w_0^{\star})=\gamma^{\star}$ and that $\gamma_1^{\star}(w_0)$ is strictly decreasing in $w_0$. See the proof of Corollary \ref{cor:unique}.
\end{remark}

\begin{figure}
  \centering
  \includegraphics[width=10cm]{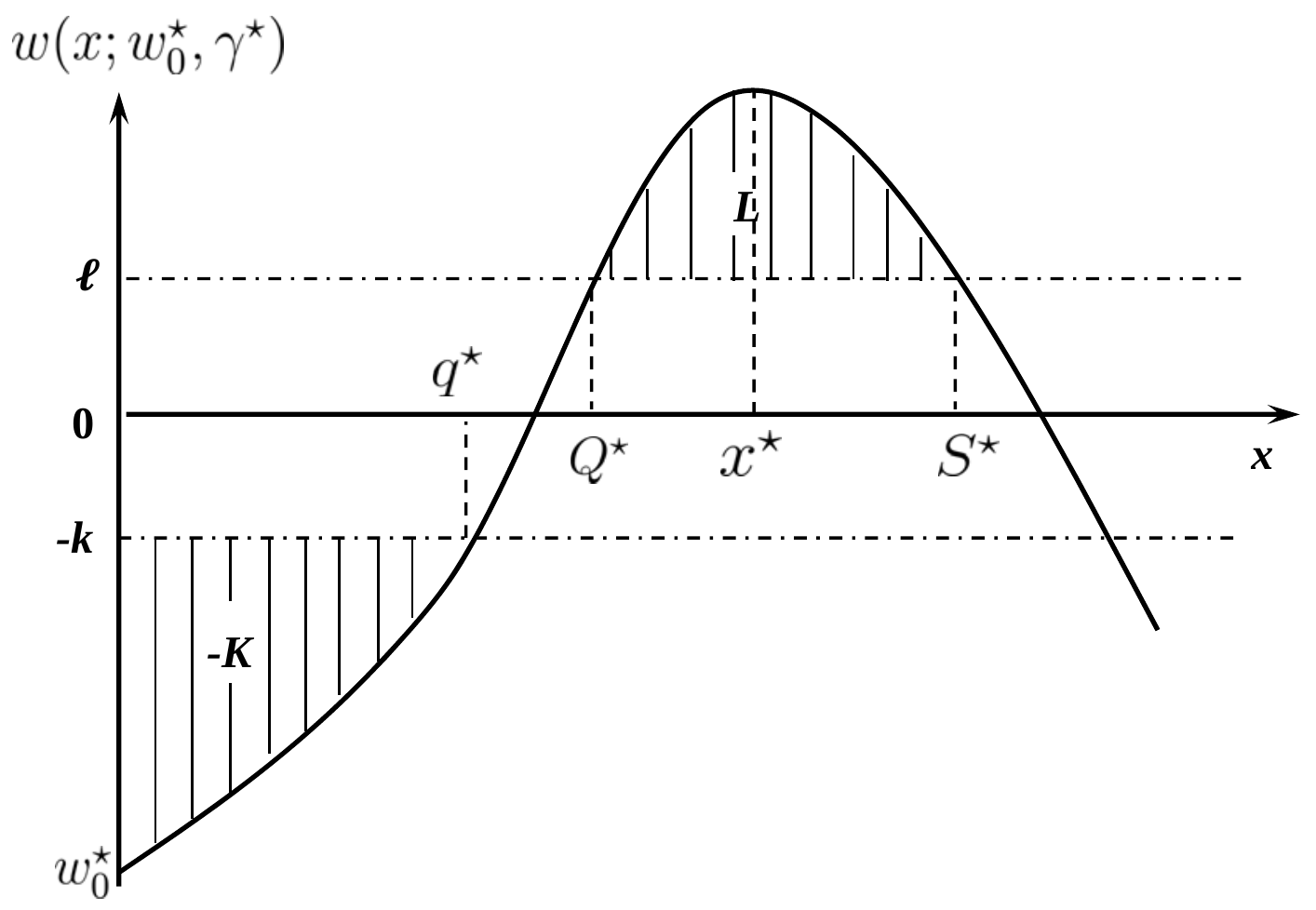}\\
  \caption{Optimal parameters and the corresponding function}\label{fig:optimal_function}
\end{figure}

Now, it remains to prove \eqref{eq:w=w=l}-\eqref{eq:gamma1=gamma2} by the following
Lemmas  \ref{lem:gamma1}-\ref{lem:gamma1=gamma2}.
Recalling  the definition of $\bar{\gamma}(w_0)$ in Lemma \ref{lem:w-property}, we first have the following lemma.
\begin{lemma}
\label{lem:gamma1}
\begin{enumerate}[$(a)$]
\item\label{item-lemma-gamma1-a}
For any $w_0<\ell$, there exists a finite number $\gamma_1(w_0)$ with
$\gamma_1(w_0)\in(\pi(w_0),\bar{\gamma}(w_0))$ such that
for any $\gamma\in(\gamma_1(w_0),\bar{\gamma}(w_0))$, there exist two unique numbers $Q(w_0,\gamma)$ and $S(w_0,\gamma)$
with $0<Q(w_0,\gamma)<x^{\star}(w_0,\gamma)<S(w_0,\gamma)$ satisfying
\[
w(Q(w_0,\gamma);w_0,\gamma)=w(S(w_0,\gamma);w_0,\gamma)=\ell.
\]
\item
For any $w_0<\ell$, there exists a unique finite number $\gamma_1^{\star}(w_0)$ with
$\gamma_1^{\star}(w_0)\in(\gamma_1(w_0),\bar{\gamma}(w_0))$ such that
\begin{equation}
\label{eq:f1=L}
f_1(w_0,\gamma_1^{\star}(w_0))=L,
\end{equation}
where $f_1(w_0,\gamma):=\int_0^{\infty}(w(x;w_0,\gamma)-\ell)^+\,\mathrm{d}x$
is strictly increasing in $\gamma\in\mathbb{R}$.
Furthermore, $\gamma_1^{\star}(w_0)$ is continuous and strictly decreasing in $w_0\in(-\infty,\ell)$.
\end{enumerate}
\end{lemma}
\begin{remark}
Letting $Q(w_0)=Q(w_0,\gamma_1^{\star}(w_0))$ and $S(w_0)=S(w_0,\gamma_1^{\star}(w_0))$,
Lemma \ref{lem:gamma1} implies \eqref{eq:w=w=l} and \eqref{eq:int=L-gamma-1}.
\end{remark}
\begin{proof}
(a) For any $w_0<\ell$, define
\[
\gamma_1(w_0)=\inf\{\gamma\in (\pi(w_0),\bar{\gamma}(w_0)): w(x^{\star}(w_0,\gamma); w_0,\gamma)\geq \ell\}.
\]
It follows from Lemma \ref{lem:w-property} \eqref{solution_x_3} and Remark \ref{rem:x-star} that
\begin{equation}
\label{equ:w_x1}
w(x^{\star}(w_0,\gamma); w_0,\gamma)=\max_{x\geq 0}w(x; w_0,\gamma),
\end{equation}
which, together with Lemma \ref{lem:w-property-1} \eqref{item-c-lemma-w} and \eqref{item-d-lemma-w},
implies that $w(x^{\star}(w_0,\gamma); w_0,\gamma)$ is strictly increasing in
both $\gamma\in(\pi(w_0),\bar{\gamma}(w_0))$ and $w_0\in(-\infty,\ell)$.
Furthermore, Lemma \ref{lem:w-property} implies that
\[
\lim_{\gamma\downarrow\pi(w_0)}w(x^{\star}(w_0,\gamma); w_0,\gamma)=w_0<\ell\quad \text{and}\quad
\lim_{\gamma\uparrow\bar{\gamma}(w_0)}w(x^{\star}(w_0,\gamma); w_0,\gamma)=\infty.
\]
Hence, $\gamma_1(w_0)$ is well defined, finite, and strictly decreasing in $w_0\in(-\infty,\ell)$.
Using the continuity of $w(x;w_0,\gamma)$ in $\gamma$, \eqref{equ:w_x1}, the definition of $\gamma_1(w_0)$, and the monotonicity of $w(x^{\star}(w_0,\gamma); w_0,\gamma)$ in $\gamma$, we also have
\begin{align*}
&w(x^{\star}(w_0,\gamma_1(w_0)); w_0,\gamma_1(w_0))=\ell\quad\text{and}\\
&w(x^{\star}(w_0,\gamma_1(w_0)); w_0,\gamma)>\ell \text{ for $\gamma\in(\gamma_1(w_0),\bar{\gamma}(w_0))$}.
\end{align*}
For $\gamma\in(\gamma_1(w_0),\bar{\gamma}(w_0))$, define
\[
Q(w_0,\gamma)=\inf\{x\geq 0: w(x; w_0,\gamma)=\ell\} \  \text{and}\ S(w_0,\gamma)=\sup\{x\geq 0: w(x; w_0,\gamma)=\ell\}.
\]
Then, it follows from Lemma~\ref{lem:w-property}~($\ref{solution_x_3}$) and $w_0<\ell$ that both $Q(w_0,\gamma)$ and $S(w_0,\gamma)$ are well defined, finite, and unique,
and so $0<Q(w_0,\gamma)<x^{\star}(w_0,\gamma)<S(w_0,\gamma)$ and $w(Q(w_0,\gamma); w_0,\gamma)=w(S(w_0,\gamma); w_0,\gamma)=\ell$.

(b) Note that $f_1(w_0,\gamma):=\int_0^{\infty}(w(x;w_0,\gamma)-\ell)^+\,\mathrm{d}x$.
Hence, Lemma \ref{lem:w-property-1} \eqref{item-c-lemma-w}, Lemma \ref{lem:w-property},
and the definitions of $Q(w_0,\gamma)$ and $S(w_0,\gamma)$ imply that
\[
f_1(w_0,\gamma)=
\begin{cases}
0 &\text {for $\gamma\in(-\infty, \gamma_1(w_0)]$},\\
\int_{Q(w_0,\gamma)}^{S(w_0,\gamma)}[w(x; w_0,\gamma)-\ell]\, \mathrm{d}x
& \text{for $\gamma\in(\gamma_1(w_0),\bar{\gamma}(w_0))$},\\
\infty & \text{for $\gamma\in[\bar{\gamma}(w_0),\infty)$},
\end{cases}
\]
and $f_1(w_0,\gamma)$ is strictly increasing in $\gamma\in(\gamma_1(w_0),\bar{\gamma}(w_0))$.
Thus, we have
\[
\lim_{\gamma\downarrow\gamma_1(w_0)}f_1(w_0,\gamma)=0\quad\text{and}\quad
\lim_{\gamma\uparrow\bar{\gamma}(w_0)}f_1(w_0,\gamma)=\infty,
\]
which, together with the continuity of $f_1(w_0,\gamma)$
in $\gamma\in(\gamma_1(w_0),\bar{\gamma}(w_0))$ (the continuity of $f_1(w_0,\gamma)$
can be implied by the continuity of $w(x;w_0,\gamma)$),
imply that there exists a unique $\gamma_1^{\star}(w_0)\in(\gamma_1(w_0),\bar{\gamma}(w_0))$
such that $f_1(w_0,\gamma_1^{\star}(w_0))=L$.

It remains to prove that  $\gamma_1^{\star}(w_0)$ is continuous and strictly decreasing in $w_0\in(-\infty,w_0)$.
First, it follows from Lemma~\ref{lem:w-property-1}~\eqref{item-d-lemma-w}
 and the definition of $f_1(w_0,\gamma)$ that $f_1(w_0, \gamma)$ is strictly increasing in
 $w_0\in(-\infty,\ell)$ when $\gamma\in (\gamma_1(w_0), \bar{\gamma}(w_0))$.
Also, $f_1(w_0, \gamma)$ is strictly increasing in $\gamma\in(\gamma_1(w_0),\bar{\gamma}(w_0))$.
Hence, $\gamma_1^{\star}(w_0)$ is strictly decreasing in $w_0\in(-\infty,\ell)$ as can be seen by noting that $\gamma_1^{\star}(w_0)\in(\gamma_1(w_0),\bar{\gamma}(w_0))$.
Finally, the continuity of $\gamma_1^{\star}(w_0)$ follows from the monotonicity of $f_1(w_0,\gamma)$
in $\gamma\in(\gamma_1(w_0),\bar{\gamma}(w_0))$, the continuity of $f_1(w_0,\gamma)$ in $w_0\in(-\infty,\ell)$, and the Implicit Function Theorem (see e.g., Theorem 1.1 in \cite{Kumagai1980JOTA}).
\end{proof}

\begin{lemma}
\label{lem:gamma2}
\begin{enumerate}[$(a)$]
\item
\label{item-a-gamma2}
For any $w_0<-k$, there exists a unique number $\gamma_2(w_0)$ with
$\gamma_2(w_0)\in(\pi(w_0),\gamma_1(w_0))$ such that
\begin{equation}
\label{eq:w=-k-lemma}
w(x^{\star}(w_0,\gamma_2(w_0)), w_0,\gamma_2(w_0))=-k,
\end{equation}
and for any $\gamma>\gamma_2(w_0)$,
there exists a unique number $q(w_0,\gamma)$ with $0<q(w_0,\gamma)<x^{\star}(w_0,\gamma) $ such that
\[
w(q(w_0,\gamma);w_0,\gamma)=-k.
\]
\item
\label{item-b-gamma2}
There exists a number $\bar{w}_0<-k$ such that
\begin{equation}
\label{eq:f2=-K-bar-w}
f_2(\bar{w}_0,\gamma_2(\bar{w}_0))=-K,
\end{equation}
and for any $w_0<\bar{w}_0$, there exists a unique number
$\gamma_2^{\star}(w_0)$ satisfying $\gamma_2^{\star}(w_0)>\gamma_2(w_0)$ and
\begin{equation}
\label{eq:f2=-K}
f_2(w_0,\gamma_2^{\star}(w_0))=-K,
\end{equation}
where $f_2(w_0,\gamma):=\int_0^{q(w_0,\gamma)}[w(x;w_0,\gamma)+k]\,\mathrm{d}x$.
Furthermore, $\gamma_2^{\star}(w_0)$ is continuous and strictly decreasing in
$w_0\in(-\infty,\bar{w}_0)$, and
\begin{equation}
\label{eq:lim-gamma=infty}
\lim_{w_0\to-\infty}\gamma_2^{\star}(w_0)=\infty.
\end{equation}
\end{enumerate}
\end{lemma}
\begin{remark}
If we let $q(w_0)=q(w_0,\gamma_2^{\star}(w_0))$, then Lemma \ref{lem:gamma2} implies
\eqref{eq:int=-K-gamma-2}.
\end{remark}
\begin{proof}
($a$) For any $w_0<-k$, define
\begin{equation}
\label{eq:gamma2(w0)}
\gamma_2(w_0)
=\inf\{\gamma\in(\pi(w_0),\bar{\gamma}(w_0)): w(x^{\star}(w_0,\gamma); w_0,\gamma)\geq -k\}.
\end{equation}
Recall the fact that $w(x^{\star}(w_0,\gamma); w_0,\gamma)$ is strictly increasing in
$\gamma\in(\pi(w_0),\bar{\gamma}(w_0))$, tends to $\infty$ as $\gamma$ goes to
$\bar{\gamma}(w_0)$ (from the argument after \eqref{equ:w_x1}), and tends to $w_0<-k$ as $\gamma$ goes to $\pi(w_0)$
(see Lemma \ref{lem:w-property}).
Then, $\gamma_2(w_0)$ is well defined and unique, and satisfies \eqref{eq:w=-k-lemma}. Moreover, $\pi(w_0)<\gamma_2(w_0)<\gamma_1(w_0)$.
For $\gamma\in(\gamma_2(w_0),\infty)$,
define
\begin{equation}
\label{eq:q(w0,gamma)}
q(w_0,\gamma)=\inf\{x\geq 0: w(x; w_0,\gamma)=-k\}.
\end{equation}
Since $w(x^{\star}(w_0,\gamma); w_0,\gamma)>-k$ for
$\gamma\in(\gamma_2(w_0),\bar{\gamma}(w_0))$ and  $w(x^{\star}(w_0,\gamma); w_0,\gamma)=\infty$
for $\gamma\in[\bar{\gamma}(w_0),\infty)$ (see Remark \ref{rem:x-star}),
we have that $q(w_0,\gamma)$ is well defined and unique, and satisfies $0<q(w_0,\gamma)<x^{\star}(w_0,\gamma)$.

($b$)  We first prove \eqref{eq:f2=-K-bar-w}.
Recall $f_2(w_0,\gamma_2(w_0))=\int_0^{q(w_0,\gamma_2(w_0))}[w(x;w_0,\gamma_2(w_0))+k]\,\mathrm{d}x$.
If we can show that  $f_2(w_0,\gamma_2(w_0))$ is strictly increasing in $w_0\in(-\infty,-k)$,
\begin{equation}
\label{eq:limf2}
\lim_{w_0\uparrow-k}f_2(w_0,\gamma_2(w_0))=0,\quad\text{and}\quad
\lim_{w_0\downarrow-\infty}f_2(w_0,\gamma_2(w_0))=-\infty,
\end{equation}
then, by defining $\bar{w}_0$ as
\[
\bar{w}_0=\sup\{w_0\leq -k: f_2(w_0,\gamma_2(w_0))=-K\},
\]
the continuity of  $f_2(w_0,\gamma)$ implies that \eqref{eq:f2=-K-bar-w} holds. (the proof of the continuity of $f_2$ is similar to that for $f_1$, and thus is omitted.) The proofs that $f_2(w_0,\gamma_2(w_0))$ is strictly increasing in $w_0\in(-\infty,-k)$
and that \eqref{eq:limf2} holds are in Appendix \ref{app:auxiliary}.

We next prove that for any $w_0<\bar{w}_0$, there exists a unique number $\gamma^{\star}_2(w_0)$ with
$\gamma^{\star}_2(w_0)>\gamma_2(w_0)$ such that \eqref{eq:f2=-K} holds.
First since  $f_2(w_0,\gamma_2(w_0))$ is strictly increasing in $w_0\in(-\infty,-k)$, we have that
for $w_0<\bar{w}_0$,
\begin{equation}
\label{eq:f2<-K}
f_2(w_0,\gamma_2(w_0))<f_2(\bar{w}_0,\gamma_2(\bar{w}_0))=-K.
\end{equation}
Furthermore, \eqref{eq:lim-w=infty-gamma} and the definition of $q(w_0,\gamma)$
in \eqref{eq:q(w0,gamma)} imply that $\lim_{\gamma\to\infty}q(w_0,\gamma)=0$ and then
\[
\lim_{\gamma\to\infty}f_2(w_0,\gamma)=0,
\]
which, together with \eqref{eq:f2<-K} and the continuity of $f_2(w_0,\gamma)$ in $\gamma$,
implies that there exists a unique $\gamma^{\star}_2(w_0)$ with
$\gamma^{\star}_2(w_0)>\gamma_2(w_0)$ such that \eqref{eq:f2=-K} holds.

We now prove that $\gamma_2^{\star}(w_0)$ is continuous and strictly decreasing in
$w_0\in(-\infty,\bar{w}_0)$.
Similar to the proof used for $\gamma_1^{\star}(w_0)$, the continuity can be obtained by the Implicit Function Theorem and is omitted.
If we can prove that $f_2(w_0,\gamma)$ is strictly increasing in $w_0\in(-\infty,\bar{w}_0)$
and $\gamma\in(\gamma_2(w_0),\infty)$ respectively, then $f_2(w_0,\gamma_2^{\star}(w_0))=-K$
implies that $\gamma_2^{\star}(w_0)$ is strictly decreasing in $w_0\in(-\infty,\bar{w}_0)$.
We next only prove  that $f_2(w_0,\gamma)$ is strictly increasing in $\gamma\in(\gamma_2(w_0),\infty)$,
and the proof for the monotonicity in $w_0\in(-\infty,\bar{w}_0)$ is very similar and is omitted.
Since $w(x;w_0,\gamma)$ is strictly increasing in $\gamma\in\mathbb{R}$ (see Lemma \ref{lem:w-property-1} \eqref{item-c-lemma-w}) and $w(x;w_0,\gamma)$ is strictly increasing in
$x\in(0,x^{\star}(w_0,\gamma))$
(see Lemma \ref{lem:w-property} \eqref{solution_x_2} and  \eqref{solution_x_3}),
$w(q(w_0,\gamma);w_0,\gamma)=-k$ implies that $q(w_0,\gamma)$ is strictly decreasing in
$\gamma\in(\gamma_2(w_0),\infty)$.
Therefore,
for any $\gamma^{\dag}$ and $\gamma^{\ddag}$ satisfying
$\gamma_2(w_0)\leq \gamma^{\dag}<\gamma^{\ddag}$,
we have
\begin{align*}
f_2(w_0,\gamma^{\dag})
&=\int_{0}^{q(w_0,\gamma^{\dag})}(w(x; w_0,\gamma^{\dag})+k)\, \mathrm{d}x
\leq \int_{0}^{q(w_0,\gamma^{\ddag})}(w(x; w_0,\gamma^{\dag})+k)\, \mathrm{d}x\\
&<\int_{0}^{q(w_0,\gamma^{\ddag})}(w(x; w_0,\gamma^{\ddag})+k)\, \mathrm{d}x
=f_2(w_0,\gamma^{\ddag}),
\end{align*}
where the first inequality follows from the fact that $q(w_0,\gamma)$ is strictly decreasing in
$\gamma\in(\gamma_2(w_0),\infty)$ and  $w(x;w_0,\gamma^{\dag})\leq -k$ for
$x\leq q(w_0,\gamma^{\dag})$, and the second inequality follows from
$w(x; w_0,\gamma)$ being strictly increasing in $\gamma\in\mathbb{R}$
(Lemma \ref{lem:w-property-1} \eqref{item-c-lemma-w})
and $q(w_0,\gamma^{\ddag})>0$.
Hence, $f_2(w_0,\gamma)$ is strictly increasing in $\gamma\in(\gamma_2(w_0),\infty)$.

We now prove \eqref{eq:lim-gamma=infty}. Suppose that it fails to hold. In that case, the fact that $\gamma_2^{\star}(w_0)$ is strictly decreasing in $w_0\in(-\infty,\bar{w}_0)$
implies that there exists a finite number $\bar{\gamma}_2^{\dag}$ such that
\[
\gamma_2^{\star}(w_0)\leq\bar{\gamma}_2^{\dag}
\quad\text{for all $w_0<\bar{w}_0$}.
\]
Since $w(x; w_0, \gamma_2^{\star}(w_0))\leq -k<0$ for $x\in [0, q(w_0, \gamma_2^{\star}(w_0))]$,
by the definition of $\pi(w)$, we have
\[
\pi(w(x; w_0, \gamma_2^{\star}(w_0)))
\geq \bar{\mu}w(x; w_0, \gamma_2^{\star}(w_0))+\underline{c}\quad
\text{for $x\in [0, q(w_0, \gamma_2^{\star}(w_0))]$},
\]
where $\underline{c}=\min_{\mu\in\mathcal{U}}c(\mu)$.
It follows from \eqref{equ:ode_new} that for $0\leq x\leq q(w_0, \gamma_2^{\star}(w_0))$,
\[
\frac{1}{2}\sigma^2w'(x; w_0,\gamma)+\bar{\mu}w(x; w_0,\gamma)+\underline{c}+h(x)
\leq \gamma,
\]
which yields that for $x\in [0, q(w_0, \gamma_2^{\star}(w_0))]$,
\begin{equation}
\label{equ:w_inequality_small}
w(x;w_0,\gamma)
\leq w_0e^{-\eta x}+\frac{2}{\sigma^2}\int_0^x\big[\gamma-\underline{c}-h(y)\big]e^{-\eta(x-y)}\, \mathrm{d}y,
\end{equation}
where $\eta=2\bar{\mu}/\sigma^2$.
Thus, letting $x=q(w_0, \gamma_2^{\star}(w_0))$ and $\gamma=\gamma_2^{\star}(w_0)$ in \eqref{equ:w_inequality_small} and using $w(q(w_0, \gamma_2^{\star}(w_0)); w_0, \gamma_2^{\star}(w_0))=-k$
and $\gamma_2^{\star}(w_0)\leq \bar{\gamma}_2^{\dag}$, we have that for all $w_0<\bar{w}_0$,
\begin{equation}
\label{eq:g>-k}
w_0e^{-\eta q(w_0, \gamma_2^{\star}(w_0))}+\frac{2}{\sigma^2}\int_0^{q(w_0, \gamma_2^{\star}(w_0))}\big[\bar{\gamma}_2^{\dag}-\underline{c}-h(y)\big]e^{-\eta(q(w_0, \gamma_2^{\star}(w_0))-y)}\, \mathrm{d}y
\geq -k.
\end{equation}
 Define
$g(x;w_0)=w_0e^{-\eta x}+\frac{2}{\sigma^2}\int_0^{x}[\bar{\gamma}_2^{\dag}-\underline{c}-h(y)]e^{-\eta(x-y)}\, \mathrm{d}y$.
Then for all $w_0<\bar{w}_0$,
\begin{equation}
\label{eq:g>-k-2}
g(q(w_0, \gamma_2^{\star}(w_0));\bar{w}_0)>g(q(w_0, \gamma_2^{\star}(w_0));w_0)\geq-k,
\end{equation}
where the last inequality follows from \eqref{eq:g>-k}.
Let
\[
\underline{q}=\inf\{x>0: g(x; \bar{w}_0)\geq -k\},
\]
and then \eqref{eq:g>-k-2} and $g(0;\bar{w}_0)=\bar{w}_0<-k$ imply that for all  $w_0<\bar{w}_0$,
\begin{equation}
\label{eq:0<underline-q}
0<\underline{q}<q(w_0, \gamma_2^{\star}(w_0)).
\end{equation}
Therefore,
\begin{align*}
f_2(w_0, \gamma^{\star}_2(w_0))
&=\int_{0}^{q(w_0,\gamma^{\star}_2(w_0))}
\big[w(x; w_0,\gamma^{\star}_2(w_0))+k\big]
\, \mathrm{d}x\\
&\leq\int_{0}^{\underline{q}}\big[w(x; w_0,\gamma^{\star}_2(w_0))+k\big]\, \mathrm{d}x\\
&\leq \int_{0}^{\underline{q}}\Big[ w_0e^{-\eta x}
+\frac{2}{\sigma^2}\int_0^x\big[\gamma^{\star}_2(w_0)-\underline{c}-h(y)\big]e^{-\eta(x-y)}dy+k\Big]\, \mathrm{d}x\\
&\leq \int_{0}^{\underline{q}}\big[g(x;w_0)+k\big]\, \mathrm{d}x,
\end{align*}
where the first inequality follows from \eqref{eq:0<underline-q}
and $w(x; w_0,\gamma^{\star}_2(w_0))+k\leq 0$ for $0\leq x\leq q(w_0, \gamma_2^{\star}(w_0))$,
the second inequality follows from \eqref{equ:w_inequality_small},
and the last inequality follows from $\gamma_2^{\star}(w_0)\leq \bar{\gamma}_2^{\dag}$.
Then, it follows from $\lim_{w_0\to-\infty}g(x;w_0)=-\infty$ for all $x\in[0,\underline{q}]$ that
\[
\lim_{w_0\to-\infty}f_2(w_0, \gamma^{\star}_2(w_0))=-\infty,
\]
which contradicts $f_2(w_0, \gamma^{\star}_2(w_0))=-K$.
Thus, we have proven \eqref{eq:lim-gamma=infty} and finished the proof.
\end{proof}

\begin{lemma}
\label{lem:gamma1=gamma2}
There exists a number $w_0^{\star}$ with $w_0^{\star}<\bar{w}_0$ such that \eqref{eq:gamma1=gamma2} holds.
\end{lemma}
\begin{proof}
We use two steps to prove this lemma.
First, we show that
\begin{equation}
\label{eq:gamma2<gamma1}
\gamma_2^{\star}(\bar{w}_0)<\gamma_1^{\star}(\bar{w}_0).
\end{equation}
Second, we prove that there exists a number $\underline{w}_0<\bar{w}_0$ such that
\begin{equation}
\label{eq:gamma2>gamma1}
\gamma_2^{\star}(\underline{w}_0)\geq\gamma_1^{\star}(\underline{w}_0).
\end{equation}
Thus, \eqref{eq:gamma2<gamma1}, \eqref{eq:gamma2>gamma1},
and the continuity of $\gamma_1^{\star}(w_0)$ and  $\gamma_2^{\star}(w_0)$
imply that there exists a number $w_0^{\star}$ such that \eqref{eq:gamma1=gamma2} holds.
We next prove \eqref{eq:gamma2<gamma1} and \eqref{eq:gamma2>gamma1}.

First, we prove \eqref{eq:gamma2<gamma1}.
It follows from \eqref{eq:w=-k-lemma} and $\bar{w}_0<-k$ that
\begin{equation}
\label{eq:w=-k<l}
w(x^{\star}(\bar{w}_0,\gamma_2^{\star}(\bar{w}_0));\bar{w}_0,\gamma_2^{\star}(\bar{w}_0))
=-k<\ell,
\end{equation}
which implies that $x^{\star}(\bar{w}_0,\gamma_2^{\star}(\bar{w}_0))$ is finite
and thus
 $\gamma_2^{\star}(\bar{w}_0)\in(\pi(w_0),\bar{\gamma}(w_0))$.
Thus,  we have
\begin{equation}
\label{eq:f1<f1}
f_1(\bar{w}_0,\gamma_2^{\star}(\bar{w}_0))=\int_0^{\infty}(w(x;\bar{w}_0,\gamma_2^{\star}(\bar{w}_0))-\ell)^+\,\mathrm{d}x=0<L=f_1(\bar{w}_0,\gamma_1^{\star}(\bar{w}_0)),
\end{equation}
where the second equality follows from Lemma \ref{lem:w-property}
\eqref{solution_x_3} and \eqref{eq:w=-k<l},
the last equality follows from \eqref{eq:f1=L} and $\bar{w}_0<-k<\ell$.
Recall that $f_1(w_0,\gamma)$ is strictly increasing in $\gamma$ in Lemma \ref{lem:gamma1}, and thus
\eqref{eq:f1<f1} yields \eqref{eq:gamma2<gamma1}.

Next we prove \eqref{eq:gamma2>gamma1}.
Suppose that it fails to hold, i.e., for all $w_0<\bar{w}_0$,
we have $\gamma_2^{\star}(w_0)<\gamma_1^{\star}(w_0)$.
Since $f_1(w_0,\gamma)$ is strictly increasing in $\gamma$,
we have that for any $w_0\leq \bar{w}_0$,
\begin{equation}
\label{eq:f1<L}
f_1(w_0,\gamma_2^{\star}(w_0))<f_1(w_0,\gamma_1^{\star}(w_0))=L.
\end{equation}
If we can show that for any fixed $x>0$,
\begin{equation}
\label{eq:lim-w=infty}
\lim_{w_0\to-\infty}w(x;w_0,\gamma_2^{\star}(w_0))=\infty.
\end{equation}
then, for any fixed pair $(x',x'')$ with $0<x'<x''$, we have that
\begin{align*}
&\liminf_{w_0\to-\infty}f_1(w_0,\gamma_2^{\star}(w_0))\\
&\quad=\liminf_{w_0\to-\infty}\int_0^{\infty}\big(w(x;w_0,\gamma_2^{\star}(w_0))-\ell\big)^+\,\mathrm{d}x\\
&\quad\geq \liminf_{w_0\to-\infty}\int_{x'}^{x''}\big(w(x;w_0,\gamma_2^{\star}(w_0))-\ell\big)^+\,\mathrm{d}x\\
&\quad\geq \liminf_{w_0\to-\infty}(x''-x')\min\big\{(w(x'';w_0,\gamma_2^{\star}(w_0))-\ell)^+, (w(x';w_0,\gamma_2^{\star}(w_0))-\ell)^+\big\}\\
&\quad=\infty,
\end{align*}
where the first inequality follows from $(w(x;w_0,\gamma_2^{\star}(w_0))-\ell)^+\geq 0$ for all $x\geq 0$,  the second inequality follows from the properties of $w(x;w_0,\gamma)$ in $x\in[0,\infty)$
for all cases in Lemma \ref{lem:w-property}, and the last equality follows from \eqref{eq:lim-w=infty}. This contradicts \eqref{eq:f1<L}, and so we get \eqref{eq:gamma2>gamma1}.

It remains to prove \eqref{eq:lim-w=infty}. Fix $x>0$.
First, \eqref{eq:lim-gamma=infty} implies that there exists a number $\bar{w}_0^{\dag}$ with
$\bar{w}_0^{\dag}<\bar{w}_0$
such that for all $y\in(0,x)$ and $w_0\in(-\infty,\bar{w}_0^{\dag})$,
\begin{equation}
\label{eq:>0}
\gamma^{\star}_2(w_0)-\pi(0)-h(y)\geq0.
\end{equation}
Choose any fixed $\hat{x}_1$ and $\hat{x}$ with $0<\hat{x}_1<\hat{x}<x$. Define
\begin{align}
 f(w_0, \hat{x}, \hat{x}_1)=&\min\left(-ke^{\xi (\hat{x}+K/(w_0+k))}, -\left(\frac{K}{\hat{x}_1}+k\right)e^{\xi (\hat{x}-\hat{x}_1)}\right) \label{def:f_w_0_hat_x}
\\
&\quad\quad\quad+\frac{2}{\sigma^2}\int_{\hat{x}_1}^{\hat{x}}\big[\gamma^{\star}_2(w_0)-\pi(0)-h(y)\big]e^{\xi(\hat{x}-y)}\,\mathrm{d}y. \nonumber
\end{align}
Note that it follows from \eqref{eq:lim-gamma=infty} that $f(w_0,\hat{x},\hat{x}_1)\to\infty$ as $w_0\to-\infty$. Hence, there exists a number $\hat{w}_0(\hat{x}, \hat{x}_1)$ with $\hat{w}_0(\hat{x}, \hat{x}_1)<\bar{w}_0^{\dag}$, such that for all $w_0$ with $w_0\leq \hat{w}_0(\hat{x}, \hat{x}_1)$, we have $f(w_0,\hat{x},\hat{x}_1)\geq 0$.

Next we show that for each $w_0$ with $w_0\leq \hat{w}_0(\hat{x}, \hat{x}_1)$,  there exists a number $\tilde{x}$ with $0<\tilde{x}\leq \hat{x}$ (here $\tilde{x}$ might depend on $w_0$) such that
\begin{equation}
\label{eq:w>0_new}
w(\tilde{x};w_0,\gamma_2^{\star}(w_0))\geq 0.
\end{equation}
Suppose that this does not hold. Then, there exists a number $w_0$ with $w_0\leq \hat{w}_0(\hat{x}, \hat{x}_1)$ such that $w(y;w_0,\gamma_2^{\star}(w_0))<0$ for all $y\in(0, \hat{x}]$. Then, \eqref{equ:inequality_w_x} implies that for all $y\in(0, \hat{x}]$,
\begin{equation}
\label{equ:w_w0_gamma_w0_new}
\frac{1}{2}\sigma^2 w'(y; w_0,\gamma^{\star}_2(w_0))-Mw(y; w_0,\gamma^{\star}_2(w_0))\geq \gamma^{\star}_2(w_0)-\pi(0)-h(y).
\end{equation}
We consider two cases: $\hat{x}_1\geq q(w_0,\gamma^{\star}_2(w_0))$ or $\hat{x}_1< q(w_0,\gamma^{\star}_2(w_0))$.

If $\hat{x}_1\geq q(w_0,\gamma^{\star}_2(w_0))$, then integrating \eqref{equ:w_w0_gamma_w0_new} from
$q(w_0,\gamma^{\star}_2(w_0))$ to $\hat{x}$ yields
\begin{align}
&w(\hat{x}; w_0,\gamma^{\star}_2(w_0))\label{equ:case1_new}\\
&\quad\geq-ke^{\xi (\hat{x}-q(w_0,\gamma^{\star}_2(w_0)))}
+\frac{2}{\sigma^2}\int_{q(w_0,\gamma^{\star}_2(w_0))}^{\hat{x}}\big[\gamma^{\star}_2(w_0)-\pi(0)-h(y)\big]e^{\xi(\hat{x}-y)}\, \mathrm{d}y \nonumber\\
&\quad\geq -ke^{\xi (\hat{x}+K/(w_0+k))}
+\frac{2}{\sigma^2}\int_{q(w_0,\gamma^{\star}_2(w_0))}^{\hat{x}}\big[\gamma^{\star}_2(w_0)-\pi(0)-h(y)\big]e^{\xi(\hat{x}-y)}\, \mathrm{d}y \nonumber\\
&\quad \geq-ke^{\xi (\hat{x}+K/(w_0+k))}
+\frac{2}{\sigma^2}\int_{\hat{x}_1}^{\hat{x}}\big[\gamma^{\star}_2(w_0)-\pi(0)-h(y)\big]e^{\xi(\hat{x}-y)}\, \mathrm{d}y, \nonumber
\end{align}
where $\xi=2M/\sigma^2$. The first inequality follows from
$w(q(w_0,\gamma^{\star}_2(w_0)); w_0,\gamma^{\star}_2(w_0))=-k$,
the second inequality follows from $-K/(w_0+k)<q(w_0,\gamma^{\star}_2(w_0))$ (noting that
$f_2(w_0,\gamma^{\star}_2(w_0))=-K\geq (w_0+k)q(w_0,\gamma_2^{\star}(w_0))$),
and the last inequality follows from \eqref{eq:>0} and $q(w_0,\gamma^{\star}_2(w_0))\leq \hat{x}_1$.

If $\hat{x}_1< q(w_0,\gamma^{\star}_2(w_0))$, then we have
\begin{align*}
-K&=f_2(w_0,\gamma^{\star}_2(w_0))=\int_0^{q(w_0,\gamma^{\star}_2(w_0))}\big[w(y; w_0,\gamma^{\star}_2(w_0))+k\big]\,\mathrm{d}y\\
&\leq \int_0^{\hat{x}_1}\big[w(y; w_0,\gamma^{\star}_2(w_0))+k\big]\,\mathrm{d}x
\leq \int_0^{\hat{x}_1}\big[w(\hat{x}_1; w_0,\gamma^{\star}_2(w_0))+k\big]\,\mathrm{d}y\\
&=\hat{x}_1\big[w(\hat{x}_1; w_0,\gamma^{\star}_2(w_0))+k\big]
\end{align*}
and thus
 $w(\hat{x}_1; w_0,\gamma^{\star}_2(w_0))\geq -K/\hat{x}_1-k$.
Integrating (\ref{equ:w_w0_gamma_w0_new}) from $\hat{x}_1$ to $\hat{x}$  yields
\begin{align}
&w(\hat{x}; w_0,\gamma^{\star}_2(w_0))\label{equ:case2_new}\\
&\quad\geq w(\hat{x}_1; w_0, \gamma^{\star}_2(w_0))e^{\xi (\hat{x}-\hat{x}_1)}+\frac{2}{\sigma^2}\int_{\hat{x}_1}^{\hat{x}}\big[\gamma^{\star}_2(w_0)-\pi(0)-h(y)\big]e^{\xi(\hat{x}-y)}\,\mathrm{d}y\nonumber\\
&\quad\geq-\left(\frac{K}{\hat{x}_1}+k\right)e^{\xi (\hat{x}-\hat{x}_1)}+\frac{2}{\sigma^2}\int_{\hat{x}_1}^{\hat{x}}\big[\gamma^{\star}_2(w_0)-\pi(0)-h(y)\big]e^{\xi(\hat{x}-y)}\,\mathrm{d}y. \nonumber
\end{align}

Combining \eqref{equ:case1_new} and \eqref{equ:case2_new}, it follows from \eqref{def:f_w_0_hat_x}, the definition of $\hat{w}_0(\hat{x}, \hat{x}_1)$, and $w_0\leq \hat{w}_0(\hat{x}, \hat{x}_1)$ that
\begin{align*}
&w(\hat{x}; w_0,\gamma^{\star}_2(w_0))\geq f(w_0, \hat{x}, \hat{x}_1)\geq 0,
\end{align*}
which contradicts $w(y;w_0,\gamma_2^{\star}(w_0))<0$ for all $y\in(0, \hat{x}]$.
Hence,  \eqref{eq:w>0_new} holds.
Then, using a proof similar to the proof of $\lim_{\gamma\to\infty}w(x; w_0,\gamma)=\infty$ in Lemma \ref{lem:w-property-1} \eqref{item-c-lemma-w} we easily have \eqref{eq:lim-w=infty}, which finishes the proof.
\end{proof}

\section{Optimality of the proposed policy}
\label{sec:optimality}

In this section, we prove Theorem~\ref{thm:mainresults} using the lower bound approach.
First, in \S \ref{sec:lower_bound-sub} we provide a lower bound for the cost under any admissible policy.
Then, in \S \ref{sec:optimality-sub} we show the cost under the proposed policy determined in \S\ref{sec:existence}
can achieve the lower bound and thus the proposed policy is an optimal one.

\subsection{Lower bound}
\label{sec:lower_bound-sub}

We give a lower bound for the cost under any admissible policy by using It\^{o}'s formula as follows.

\begin{proposition}
\label{prop:lower_bound}
Suppose that there exists a constant $\gamma$ and a function $f(\cdot)\in \mathscr{C}^1(\mathbb{R}^+)$
with absolutely continuous and bounded derivative $f'$ and continuous second derivative $f''$ at all but a finite number of points,
satisfying
\begin{equation}\label{equ:drift}
\frac{1}{2}\sigma^2 f''(x)+\min_{\mu\in\mathcal{U}}(\mu f'(x)+c(\mu))+h(x)\geq\gamma
\quad \text{for all $x\in\mathbb{R}^+$ at which $f''$ exists,}
\end{equation}
with
\begin{align}
&f(x)\leq f(y)+K+k(y-x) \quad \text{for $0\leq x<y$},\label{equ:impulse_up}\\
&f(x)\leq f(y)+L+\ell(x-y) \quad \text{for $0\leq y<x$}.\label{equ:impulse_down}
\end{align}
Moreover, suppose that $f$ is bounded below, i.e., there exists a finite number $f_{LB}$ such that
\begin{equation}\label{equ:f_bound_below}
f(x)\geq f_{LB}, \text{ for all }x\geq 0.
\end{equation}
Then $\mathcal{C}(x,\phi)\geq \gamma$ for any admissible policy $\phi=(\bm{\mu},\bm{Y})$ and initial state $x\in\mathbb{R}^+$.
\end{proposition}
\begin{proof}
It follows from It\^{o}'s formula (see, e.g., Proposition 1 in \cite{OrmeciDaiVandeVate2008}) and \eqref{equ:f_bound_below} that
\begin{align}
\label{equ:dynkin_formula}
\mathbb{E}_x[f(X_t)]&=\mathbb{E}_x[f(X_0)]+\mathbb{E}_x\Big[\int_0^t\Big(\frac{1}{2}\sigma^2f''(X_s)+\mu_sf(X_s)\Big)\, \mathrm{d}s\Big]\\
&\quad+\sum_{i=1}^2\mathbb{E}_x\Big[\sum_{n=1}^{N^i_t}\left(f(X_{\tau^i_n})-f(X_{\tau^i_n-})\right)\Big]. \nonumber
\end{align}
Notice that (\ref{equ:drift}) implies that $1/2\sigma^2 f''(x)+\mu(x)f'(x)+c(\mu)+h(x)\geq \gamma$ for all $x\in\mathbb{R}^+$ and $\mu\in\mathcal{U}$.
Furthermore, \eqref{equ:impulse_up} and \eqref{equ:impulse_down} imply that for each $n\geq 0$, we have
$f(X_{\tau^1_n})-f(X_{\tau^1_n-})\geq -(K+k\xi_n^1)$ and $f(X_{\tau^2_n})-f(X_{\tau^2_n-})\geq -(L+\ell\xi_n^2)$.
Therefore, we have
\begin{align*}
&\mathbb{E}_x[f(X_t)]\\
&\quad\geq\mathbb{E}_x[f(X_0)]+\mathbb{E}_x\Big[\int_0^t\big(\gamma-c(\mu_s)-h(X_s)\big)\,\mathrm{d}s\Big]
-\mathbb{E}_x\Big[\sum_{n=1}^{N^1_t}\big(K+k\xi_n^1\big)\Big]\\
&\qquad-\mathbb{E}_x\Big[\sum_{n=1}^{N^2_t}\big(L+\ell\xi_n^2\big)\Big]\\
&\quad=\mathbb{E}_x[f(X_0)]+\gamma t-\mathbb{E}_x\Big[\int_0^t \big(h(X_s)+c(\mu_s)\big)\, \mathrm{d}s
+\sum_{n=1}^{N^1_t}\big(K+k\xi_n^1\big)+\sum_{n=1}^{N^2_t}\big(L+\ell\xi_n^2\big)\Big].
\end{align*}
Dividing both sides of the above inequality by $t$ and letting $t\to\infty$ gives
\begin{equation}\label{inequ:c_u_y}
\mathcal{C}(x,\phi)\geq \gamma+\limsup_{t\to\infty}\frac{\mathbb{E}_x[f(X_t)]}{t}.
\end{equation}
It follows from (\ref{equ:f_bound_below}) that $\limsup_{t\to\infty}\mathbb{E}_x[f(X_t)]/t\geq 0$. Hence, (\ref{inequ:c_u_y}) yields $\mathcal{C}(x,\phi)\geq \gamma$.
\end{proof}

\subsection{Optimality of proposed policy $\phi^{\star}$}
\label{sec:optimality-sub}

In this subsection, Proposition \ref{prop:cost} characterizes the cost under any policy
$\phi=\{(0, q, Q, S), \{\mu(x)\in\mathcal{U}, x\in [0, S]\}\}$,
and then Theorem \ref{thm:mainresults} is proved by showing that
the cost $\gamma^{\star}$ under the proposed policy $\phi^{\star}$ can achieve the lower bound in Proposition \ref{prop:lower_bound}.
Finally, Corollary \ref{cor:unique} is proven.

\begin{proposition}
\label{prop:cost}
Consider a policy $\phi=\{(0, q, Q, S), \{\mu(x)\in\mathcal{U}, x\in [0, S]\}\}$ with $0<q\leq Q<S$.
Suppose that there exists a constant $\gamma$ and a  twice continuously differentiable function: $V: [0, S]\to\mathbb{R}$ satisfying
\begin{align}
 & \frac{1}{2}\sigma^2 V''(x)+\mu(x)V'(x)+c(\mu(x))+h(x)=\gamma
 \quad\text{for $0\leq x\leq S$,} \label{equ:poisson}
\end{align}
with boundary conditions
\begin{align}
&V(0)=V(q)+K+kq, \label{equ:bound_1}\\
&V(S)=V(Q)+L+\ell(S-Q). \label{equ:bound_2}
\end{align}
Then, the average cost $\mathcal{C}(x,\phi)=\gamma$ for any initial state $x\in\mathbb{R}^+$.
\end{proposition}

\begin{proof}
If the initial state $x\geq S$, there will be a one-time control to bring it to $Q$ and thus the state will stay in $[0, S]$ forever under the policy $\phi$.
The one-time finite control cost can be ignored in the long-run average cost,
thus it suffices to consider the case that the initial state $x\in[0, S]$.

Since $V$ is twice continuously differentiable on $[0, S]$, it has a bounded derivative on $[0, S]$.
Furthermore,  it follows from \eqref{equ:bound_1} and \eqref{equ:bound_2} that
under policy $\phi$, $V(X_{\tau^1_n})-V(X_{\tau^1_n-})=-(K+k\xi_n^1)$ and
$V(X_{\tau^2_n})-V(X_{\tau^2_n-})=-(L+\ell\xi_n^2)$.
Since $0\leq X_t\leq S$ for all $t>0$ under policy $\phi$,
it follows from \eqref{equ:dynkin_formula} and (\ref{equ:poisson}) that
\begin{align}
\label{eq:E-V}
\mathbb{E}_x[V(X_t)]
&=\mathbb{E}_x[V(X_0)]+\mathbb{E}_x\Big[\int_0^t \big(\frac{1}{2}\sigma^2V''(X_s)+\mu(X_s)V(X_s)\big)
\, \mathrm{d}s\Big]\\
&\quad+\sum_{i=1}^2\mathbb{E}_x\Big[\sum_{n=1}^{N^i_t}\Big(V(X_{\tau^i_n})-V(X_{\tau^i_n-})\Big)\Big]\nonumber\\
&=\mathbb{E}_x[V(X_0)]+\mathbb{E}_x\Big[\int_0^t(\gamma-c(\mu(X_s))-h(X_s))\, \mathrm{d}s\Big]\nonumber\\
&\quad-\mathbb{E}_x\Big[\sum_{n=1}^{N^1_t}\big(K+k\xi_n^1\big)\Big]
-\mathbb{E}_x\Big[\sum_{n=1}^{N^2_t}\big(L+\ell\xi_n^2\big)\Big]\nonumber\\
&=\mathbb{E}_x[V(X_0)]+\gamma t\nonumber\\
&\quad-\mathbb{E}_x\Big[\int_0^t \big(h(X_s)+c(\mu(X_s))\big)\, \mathrm{d}s
+\sum_{n=1}^{N^1_t}\big(K+k\xi_n^1\big)+\sum_{n=1}^{N^2_t}\big(L+\ell\xi_n^2\big)\Big]\nonumber
\end{align}
Note that $\min_{0\leq x\leq S}V(x)\leq V(X_t)\leq \max_{0\leq x\leq S}V(x)$,
which implies
\[
\lim_{t\to\infty}\mathbb{E}_x[V(X_t)]/t=0.
\]
Dividing both sides of \eqref{eq:E-V} by $t$ and letting $t\to\infty$, we have $\mathcal{C}(x,\phi)=\gamma$.
\end{proof}


We are now ready to prove Theorem \ref{thm:mainresults}.
Recalling the solution to \eqref{eq:ODE}-\eqref{eq:w(Q)=w(S)=l} in Theorem \ref{thm:existence}, i.e., $(w^{\star},\{q^{\star},Q^{\star},S^{\star},\gamma^{\star}\})$,
we define
\[
f^{\star}(x)=
\begin{cases}
\int_0^x w^{\star}(y)\, \mathrm{d}y & \text{for $0\leq x\leq S^{\star}$},\\
\int_0^{S^{\star}} w^{\star}(y)\, \mathrm{d}y+\ell(x-S^{\star}) & \text{for $x>S^{\star}$}.
\end{cases}
\]
\begin{proof}[Proof of Theorem~\ref{thm:mainresults}]
Since $\phi^{\star}=\{(0, q^{\star}, Q^{\star}, S^{\star}), \{\mu^{\star}(x)\in\mathcal{U}, x\in [0, S^{\star}]\}\}$ satisfies \eqref{eq:ODE}-\eqref{eq:int=L} in Theorem \ref{thm:existence},
letting $V(x)=f^{\star}(x)$ in Proposition \ref{prop:cost},
we have that  $\gamma^{\star}$ is the average cost under $\phi^{\star}$. If we can prove that $f^{\star}$ and $\gamma^{\star}$ satisfy the conditions in
Proposition \ref{prop:lower_bound}, then $\mathcal{C}(x,\phi)\geq \gamma^{\star}$
for any admissible policy $\phi\in\Pi$ and any initial state $x\in\mathbb{R}^+$,
and thus $\gamma^{\star}$ is the optimal average cost and $\phi^{\star}$ is an optimal policy.

It remains to check the conditions in Proposition \ref{prop:lower_bound}.
First,  it follows from \eqref{eq:w(Q)=w(S)=l} and the definition of $f^{\star}$ that $f^{\star}$ has an absolutely continuous and bounded derivative $\mathrm{d}f^{\star}(x)/\mathrm{d} x$.  Moreover, it follows from \eqref{eq:ODE} that $f^{\star}$ also has a continuous second derivative at all points in $\mathbb{R}^+$ except maybe $S^{\star}$.  Moreover, it follows from the definition of $f^{\star}$ and $\ell>0$ that $f^{\star}(x)\geq \min_{0\leq x\leq S^{\star}}f(x)$ for all $x\geq 0$. Hence,  \eqref{equ:f_bound_below} holds.

Next, we check \eqref{equ:drift}.  It follows from (\ref{eq:ODE})  and the definition of $f^{\star}$ that (\ref{equ:drift}) holds for $x\in [0, S^{\star}]$.
For $x>S^{\star}$, we have
 \begin{align*}
 &\frac{1}{2}\sigma^2 \frac{\mathrm{d}^2f^{\star}(x)}{\mathrm{d} x^2}
 +\min_{\mu\in\mathcal{U}}\Big(\mu \frac{\mathrm{d}f^{\star}(x)}{\mathrm{d} x}+c(\mu)\Big)+h(x)\\
 &\quad=\min_{\mu\in\mathcal{U}}\big(\mu \ell+c(\mu)\big)+h(x)\\
 &\quad>\min_{\mu\in\mathcal{U}}\big(\mu \ell+c(\mu)\big)+h(S^{\star})\\
 &\quad\geq \frac{1}{2}\sigma^2 \frac{\mathrm{d}w^{\star}(S^{\star})}{\mathrm{d} x}
 +\min_{\mu\in\mathcal{U}}\Big(\mu w^{\star}(S^{\star})+c(\mu)\Big)+h(S^{\star})\\
 &\quad=\gamma^{\star},
 \end{align*}
 where the first equality follows from $\mathrm{d}f^{\star}(x)/\mathrm{d} x=\ell$ and
 $\mathrm{d}^2f^{\star}(x)/\mathrm{d} x^2=0$, the first inequality follows from
 Assumption \ref{assum-h} and $x>S^{\star}$, the second inequality follows from
 \eqref{eq:w(Q)=w(S)=l} and the fact that $w^{\star}(x)$ is strictly decreasing in $x$ at $S^{\star}$,
 and the second equality follows from \eqref{eq:ODE} with $x=S^{\star}$.

We next check \eqref{equ:impulse_up}.
For $0\leq x<y$, the proof is divided into three cases: $x>S^{\star}$, $0\leq x<y\leq S^{\star}$,
or $0\leq x\leq S^{\star}<y$.
If $x>S^{\star}$, we have $f^{\star}(y)+K+k(y-x)-f^{\star}(x)=(\ell+k)(y-x)+K>0$.
If $0\leq x<y\leq S^{\star}$, we have
\begin{equation}
\label{eq:f>0}
 f^{\star}(y)+K+k(y-x)-f^{\star}(x)
 =\int_x^y(w^{\star}(z)+k)\, \mathrm{d}z+K
 \geq \int_0^{q^{\star}}(w^{\star}(z)+k)\, \mathrm{d}z+K=0,
 \end{equation}
where the inequality follows from $w^{\star}(x)\leq -k$ for $0\leq x\leq q^{\star}$ and $w^{\star}(x)\geq -k$ for $q^{\star}\leq x\leq S^{\star}$ (see Theorem~\ref{thm:existence}), and the last equality follows from \eqref{eq:int=-K}.
When $0\leq x\leq S^{\star}<y$, we have
\[
f^{\star}(y)+K+k(y-x)-f^{\star}(x)
=f^{\star}(S^{\star})+K+k(S^{\star}-x)-f^{\star}(x)+(\ell+k)(y-S^{\star})
\geq 0,
\]
where the equality follows from $f^{\star}(y)=f^{\star}(S^{\star})+\ell(y-S^{\star})$
and the inequality follows from $f(S^{\star})+K+k(S^{\star}-x)-f(x)\geq 0$
(using \eqref{eq:f>0} with $y=S^{\star}$).

Finally, we check that  \eqref{equ:impulse_down} holds for all $0\leq y<x$.
This proof is also divided into three cases: $y>S^{\star}$, $0\leq y<x\leq S^{\star}$,
or $0\leq y\leq S^{\star}<x$.
If $y>S^{\star}$, we have $f(y)+L+\ell(x-y)-f(x)=L>0$.
If $0\leq y<x\leq S^{\star}$, we have
\begin{equation}
\label{eq:f>0-2}
f(y)+L+\ell(x-y)-f(x)
=-\int_y^x(w^{\star}(z)-\ell)\, \mathrm{d}z+L
\geq -\int_{Q^{\star}}^{S^{\star}}(w^{\star}(z)-\ell)\, \mathrm{d}z+L
=0,
\end{equation}
 where the inequality is due to $w^{\star}(x)\leq \ell$ for $0\leq x\leq Q^{\star}$ and $w^{\star}(x)\geq \ell$ for $Q^{\star}\leq x\leq S^{\star}$ (see Theorem~\ref{thm:existence}), and the last equality follows from \eqref{eq:int=L}.
If $0\leq y\leq S^{\star}<x$, we have
\[
f(y)+L+\ell(x-y)-f(x)=f(y)+L+\ell(S^{\star}-y)-f(S^{\star})\geq 0,
\]
where the inequality follows from \eqref{eq:f>0-2} by letting $x=S^{\star}$.
\end{proof}

Finally, we prove Corollary \ref{cor:unique}.
\begin{proof}[Proof of Corollary \ref{cor:unique}]
It follows from Theorem 2 that $\gamma^{\star}$ is the optimal cost and thus it is unique in Theorem \ref{thm:existence}. Recalling $\gamma^{\star}=\gamma_1(w_0^{\star})$ and that $\gamma_1(w_0)$ is strictly decreasing in $w_0\in(-\infty,\ell)$ (see Lemma \ref{lem:gamma1}), we see that $w_0^{\star}$ must be unique
and so $q^{\star}, Q^{\star}, S^{\star}$ are all unique (see Remark \ref{rem:uniqueness}). Finally, the uniqueness of $w^{\star}$ follows from Lemma \ref{lem:w-existence} \eqref{item-a-lemma-w}.
\end{proof}

\section{Concluding remarks}
\label{sec:concluding}

In this paper, we considered a  joint drift rate control and impulse control problem for a Brownian inventory/production system with the objective of minimizing the long-run average cost. We proved that an optimal policy has the $\{(0,q^{\star},Q^{\star},S^{\star}),\{\mu^{\star}(x): x\in[0,S^{\star}]\}\}$ structure by using the lower bound approach. The existence of the optimal policy parameters was shown by solving a free boundary problem, which is crucial in this paper. We provided a roadmap to solve this free boundary problem, which we believe can be used in other control problems with more general processes.

We next discuss one extension to our model. We have assumed in this paper that the inventory level must be nonnegative. In the analysis, the nonnegative constraint is needed to solve the ODE with initial condition $w(0)=w_0$.
In fact, backlog is also allowed in many inventory problems.
In the case of backlog, a similar result, such as the optimality of a $\{(d^{\star},D^{\star},U^{\star},u^{\star}),\{\mu^{\star}(x): x\in[d^{\star},u^{\star}]\}\}$ policy with control band policy $(d^{\star},D^{\star},U^{\star},u^{\star})$ and drift rate control $\{\mu^{\star}(x): x\in[d^{\star},u^{\star}]\}$, is expected. However, analyzing the ODE from $x=0$ may be inappropriate. We may start from a point sufficiently small or large, so that all policy parameters can fall in the same side of this point, just like what have done in this paper.

There are several directions worthy of future research.
First, future analysis could discuss a system whose netput inventory level is given by a compound Poisson demand process plus a Brownian motion with changeable drift.
\cite{BenkheroufBensoussan2009} and \cite{BensoussanLiuSethi2005} consider this model when the drift rate is a constant and find a solution to the corresponding quasi-variational inequality by using the Laplace transform. However, in the presence of drift rate control, the previous analysis method can not work directly and a new analytical method must be introduced to handle this problem.
Second, future research could also consider the problem of minimizing the discounted total cost rather than minimizing long-run average cost.
In the discounted cost case, the ODE in dynamic programming equation will no longer be one-order, but will have a two-order form as follows
\[
\frac{1}{2}\sigma^2 V''(x)+\pi(V'(x))-\alpha V(x)+h(x)=0,
\]
where $V(\cdot)$ is the optimal value function and $\alpha$ is the discount rate. With such an ODE, the proof for the existence of optimal policy parameters requires an argument different from the one presented here.

\appendix
\section{Proof of Lemma \ref{lem:property_pi_mu}}
\label{app:proof-lemma-1}
\begin{proof}
Since $\mu w+c(\mu)$ is linear in $w$ for any fixed $\mu\in\mathcal{U}$,
we know that $\pi(w)$ is concave in $w\in\mathbb{R}$ because the concavity is preserved under the minimization operator; see e.g., \S3.2.3 in
\cite{BoydVandenberghe2004}.

Fixing any $w_1<w_2$, we have
\begin{align*}
\abs{\pi(w_2)-\pi(w_1)}
&=\abs{\min_{\mu\in\mathcal{U}}(\mu w_2+c(\mu))-\min_{\mu\in\mathcal{U}}(\mu w_1+c(\mu))}\\
&\leq\max_{\mu\in\mathcal{U}}\abs{(\mu w_2+c(\mu))-(\mu w_1+c(\mu))}\\
&=\max_{\mu\in\mathcal{U}}\abs{\mu(w_2-w_1)}\\
&=\max\{\abs{\underline{\mu}}, \abs{\bar{\mu}}\}(w_2-w_1)\\
&=M(w_2-w_1),
\end{align*}
and thus $\pi(w)$ is Lipschitz continuous.

Denote $\mu_1=\mu(w_1)$ and $\mu_2=\mu(w_2)$. By the definition of $\mu(w)$, we have
\begin{equation*}
\mu_1w_1+c(\mu_1)\leq \mu_2w_1+c(\mu_2) \quad \text{and}\quad
\mu_2w_2+c(\mu_2)\leq \mu_1w_2+c(\mu_1).
\end{equation*}
Summing these two inequalities, we have $(\mu_2-\mu_1)(w_2-w_1)\leq 0$ and thus $\mu_2\leq \mu_1$ if $w_1<w_2$.
Hence, $\mu(w)$ is decreasing in $w$.
\end{proof}

\section{Auxiliary proof for Lemma \ref{lem:gamma2}}
\label{app:auxiliary}
\begin{proof}
First, we prove that $f_2(w_0,\gamma_2(w_0))$ is strictly increasing in $w_0\in(-\infty,-k)$.
Let $w_0^{\dag}$ and $w_0^{\ddag}$ be any two numbers satisfying  $w_0^{\dag}<w_0^{\ddag}<-k$.
If we can prove
\begin{align}
&q(w_0^{\dag},\gamma_2(w_0^{\dag}))>q(w_0^{\ddag},\gamma_2(w_0^{\ddag}))
\quad\text{and}\label{eq:q>q}\\
&w(x;w_0^{\dag},\gamma_2(w_0^{\dag}))<w(x;w_0^{\ddag},\gamma_2(w_0^{\ddag}))
\quad\text{for all $x\in[0,q(w_0^{\ddag},\gamma_2(w_0^{\ddag}))]$},
\label{eq:w<w}
\end{align}
then
\begin{align*}
f_2(w_0^{\dag},\gamma_2(w_0^{\dag}))
&=\int_{0}^{q(w_0^{\dag},\gamma_2(w_0^{\dag}))}
\big[w(x;w_0^{\dag},\gamma_2(w_0^{\dag}))+k\big]\,\mathrm{d}x\\
&<\int_{0}^{q(w_0^{\ddag},\gamma_2(w_0^{\ddag}))}
\big[w(x;w_0^{\dag},\gamma_2(w_0^{\dag}))+k\big]\,\mathrm{d}x\\
&<\int_{0}^{q(w_0^{\ddag},\gamma_2(w_0^{\ddag}))}
\big[w(x;w_0^{\ddag},\gamma_2(w_0^{\ddag}))+k\big]\,\mathrm{d}x\\
&=f_2(w_0^{\ddag},\gamma_2(w_0^{\ddag})),
\end{align*}
where the first inequality follows from \eqref{eq:q>q}
and $w(x;w_0^{\dag},\gamma_2(w_0^{\dag}))+k<0$ for all
$x\in[0,q(w_0^{\dag},\gamma_2(w_0^{\dag}))]$ (see \eqref{eq:q(w0,gamma)}),
and the second inequality follows from \eqref{eq:w<w}.
Thus, $f_2(w_0,\gamma_2(w_0))$ is strictly increasing in $w_0\in(-\infty,-k)$.

We next prove \eqref{eq:q>q}.
Recalling the definition of $x^{\star}(w_0,\gamma)$ in Lemma \ref{lem:w-property} \eqref{solution_x_3}
and \eqref{eq:w=-k-lemma}, we have
\[
\max_{x\geq0} w(x;w_0,\gamma_2(w_0))=-k,
\]
which, together with Lemma \ref{lem:w-property-1} \eqref{item-c-lemma-w} and \eqref{item-d-lemma-w},
implies that
\begin{equation}
\label{eq:gamma2-decreasing}
\gamma_2(w_0)
\text{ is strictly decreasing in $w_0\in(-\infty,-k)$}.
\end{equation}
Furthermore, it follows from \eqref{eq:w=-k-lemma} and \eqref{eq:q(w0,gamma)} that
\begin{equation}
\label{eq:w_q_w_gamma}
w(q(w_0,\gamma_2(w_0));w_0,\gamma_2(w_0))
=w(x^{\star}(w_0,\gamma_2(w_0));w_0,\gamma_2(w_0))=-k,
\end{equation}
which together with
the properties of $w$, yields
\[
q(w_0,\gamma_2(w_0))=x^{\star}(w_0,\gamma_2(w_0)).
\]
Then,
\[
w'(q(w_0,\gamma_2(w_0));w_0,\gamma_2(w_0))=w'(x^{\star}(w_0,\gamma_2(w_0));w_0,\gamma_2(w_0))
=0.
\]
Thus, taking $\gamma=\gamma_2(w_0)$ and $x=q(w_0,\gamma_2(w_0))$ in \eqref{equ:ode_new},
we have
\begin{equation}
\label{equ:q_w_gamma}
\pi(-k)+h(q(w_0,\gamma_2(w_0))=\gamma_2(w_0),
\end{equation}
which, together with the monotonicity of $h(x)$ and $\gamma_2(w_0)$ (see \eqref{eq:gamma2-decreasing}),
implies \eqref{eq:q>q}.

We next prove \eqref{eq:w<w}.  Lemma \ref{lem:w-property} \eqref{solution_x_3} and \eqref{eq:q>q}
imply that
\begin{align}
\label{eq:w<w-q}
w(q(w_0^{\ddag},\gamma_2(w_0^{\ddag}));w_0^{\dag},\gamma_2(w_0^{\dag}))
&<w(q(w_0^{\dag},\gamma_2(w_0^{\dag}));w_0^{\dag},\gamma_2(w_0^{\dag}))\\
&=-k\nonumber\\
&=w(q(w_0^{\ddag},\gamma_2(w_0^{\ddag}));w_0^{\ddag},\gamma_2(w_0^{\ddag}))\nonumber
\end{align}
Define $f_3(x)=w(x;w_0^{\ddag},\gamma_2(w_0^{\ddag}))-w(x;w_0^{\dag},\gamma_2(w_0^{\dag}))$.
Then
\[
f_3(0)=w_0^{\ddag}-w_0^{\dag}>0\quad\text{and}\quad
f_3(q(w_0^{\ddag},\gamma_2(w_0^{\ddag})))>0,
\]
where the inequality follows from \eqref{eq:w<w-q}.
Suppose \eqref{eq:w<w} does not hold. Then, by the continuity of $f_3(\cdot)$, there exists a number
$x_1^{\dag}\in(0,q(w_0^{\ddag},\gamma_2(w_0^{\ddag})))$ such that
$f_3(x_1^{\dag})=0$.
It follows from  \eqref{equ:ode_new} that
\begin{align*}
&w'(x_1^{\dag};w_0^{\dag},\gamma_2(w_0^{\dag}))
+\pi(w(x_1^{\dag};w_0^{\dag},\gamma_2(w_0^{\dag})))+h(x_1^{\dag})=\gamma_2(w_0^{\dag})\quad
\text{and}\\
&w'(x_1^{\dag};w_0^{\ddag},\gamma_2(w_0^{\ddag}))
+\pi(w(x_1^{\dag};w_0^{\ddag},\gamma_2(w_0^{\ddag})))+h(x_1^{\dag})=\gamma_2(w_0^{\ddag}),
\end{align*}
which, together with $f_3(x_1^{\dag})=0$, imply that
\begin{equation}
\label{f3'<0}
f_3'(x_1^{\dag})=\gamma_2(w_0^{\ddag})-\gamma_2(w_0^{\dag})<0,
\end{equation}
where the inequality follows from \eqref{eq:gamma2-decreasing}.
Hence, there exists an $x_2^{\dag}\in(x_1^{\dag},q(w_0^{\ddag},\gamma_2(w_0^{\ddag})))$ such that
$f_3(x_2^{\dag})<0$, which together with $f_3(q(w_0^{\ddag},\gamma_2(w_0^{\ddag})))>0$,
implies that there must exist an $x_3^{\dag}\in(x_2^{\dag},q(w_0^{\ddag},\gamma_2(w_0^{\ddag})))$
such that
\begin{equation}
\label{eq:f3'>0}
f_3(x_3^{\dag})=0\quad\text{and}\quad f_3'(x_3^{\dag})>0.
\end{equation}
Furthermore, considering $f_3(x_3^{\dag})=0$ in a way, similar to the analysis of \eqref{f3'<0},  the first part of \eqref{eq:f3'>0} implies
$f_3'(x_3^{\dag})<0$, which contradicts the second part of \eqref{eq:f3'>0}.
Thus, we have proven \eqref{eq:w<w} and have finished the proof
that $f_2(w_0,\gamma_2(w_0))$ is strictly increasing in $w_0\in(-\infty,-k)$.

Letting $w_0\to-k$,  it follows from the definition of $q(w_0,\gamma)$ that
$q(w_0,\gamma_2(w_0))\to0$ and thus
$\lim_{w_0\to-k}f_2(w_0,\gamma_2(w_0))=0$.

We next prove the second part of \eqref{eq:limf2}.
First, we show that
\begin{equation}
\label{eq:lim-gamma2=infty}
\lim_{w_0\to-\infty}\gamma_2(w_0)=\infty.
\end{equation}
If it fails to hold, then \eqref{eq:gamma2-decreasing} implies that
there exists a finite number $\bar{\gamma}_2^{\ddag}$ such that
$\gamma_2(w_0)\leq\bar{\gamma}_2^{\ddag}$ for all $w_0<-k$.
Hence, it follows from \eqref{equ:q_w_gamma} and Assumption~\ref{assum-h} that there exists a finite number $\bar{q}^{\ddag}$
such that $q(w_0,\gamma_2(w_0))\leq \bar{q}^{\ddag}$ for all $w_0<-k$.
Therefore, we have that for all $w_0<-k$,
\begin{equation}
\label{inequality:w_gamma2}
w(x^{\star}(w_0,\gamma_2(w_0)); w_0, \gamma_2(w_0))
=\max_{0\leq x\leq \bar{q}^{\ddag}}w(x; w_0, \gamma_2(w_0))
\leq \max_{0\leq x\leq \bar{q}^{\ddag}}w(x; w_0, \bar{\gamma}_2^{\ddag}),
\end{equation}
where the equality follows from
$x^{\star}(w_0,\gamma_2(w_0))=q(w_0,\gamma_2(w_0))\leq \bar{q}^{\ddag}$
(see \eqref{eq:w=-k-lemma} and \eqref{eq:q(w0,gamma)})
and Lemma~\ref{lem:w-property} \eqref{solution_x_3},
and the inequality follows from $\gamma_2(w_0)\leq\bar{\gamma}_2^{\ddag}$ and
Lemma~\ref{lem:w-property-1} \eqref{item-c-lemma-w}.
Hence, it follows from \eqref{eq:w_q_w_gamma} that
\[\max_{0\leq x\leq \bar{q}^{\ddag}}w(x; w_0, \bar{\gamma}_2^{\ddag})\geq -k
\]
for all $w_0<-k$. However, it follows from
Lemma~\ref{lem:w-property-1} \eqref{item-d-lemma-w} that
$w(x; w_0, \bar{\gamma}_2^{\ddag})\to-\infty$
as $w_0\to-\infty$ for any $x\in[0,\bar{q}^{\ddag}]$, which is a contradiction.
Therefore, \eqref{eq:lim-gamma2=infty} holds.

Using  \eqref{equ:q_w_gamma}
and $\lim_{w_0\to-\infty}\gamma_2(w_0)=\infty$,
we have  $\lim_{w_0\to-\infty}q(w_0,\gamma_2(w_0))=\infty$.
If the second part of \eqref{eq:limf2} fails to hold, then
the fact that $f_2(w_0,\gamma_2(w_0))$ is strictly increasing in $w_0\in(-\infty,-k)$
implies that there exists a number $\underline{f}$ such that $f_2(w_0, \gamma_2(w_0))\geq\underline{f}$ for all $w_0<-k$.
For any fixed pair $(x_4^{\dag},x_5^{\dag})$ with  $0<x_4^{\dag}<x_5^{\dag}$,
$\lim_{w_0\to-\infty}q(w_0,\gamma_2(w_0))=\infty$ implies that
there exists a $w_0^{\sharp}$ such that for any $w_0<w_0^{\sharp}$,
\[
q(w_0,\gamma_2(w_0))>x_5^{\dag}>x_4^{\dag}.
\]
Then, it follows from Lemma~\ref{lem:w-property} ($\ref{solution_x_3}$) that
\begin{align*}
\underline{f}&\leq f_2(w_0,\gamma_2(w_0))
=\int_0^{q(w_0,\gamma_2(w_0))}\big[w(x; w_0,\gamma_2(w_0))+k\big]\, \mathrm{d}x\\
&\leq \int_0^{x_4^{\dag}}\big[w(x_4^{\dag}; w_0,\gamma_2(w_0))+k\big]\, \mathrm{d}x
=x_4^{\dag}\big[w(x_4^{\dag}; w_0,\gamma_2(w_0))+k\big],
\end{align*}
which implies that $w(x_4^{\dag}; w_0,\gamma_2(w_0))\geq\underline{f}/x_4^{\dag}-k$.
Furthermore, it follows from \eqref{equ:inequality_w_x} and
$w(q(w_0,\gamma_2(w_0)); w_0,\gamma_2(x_0)))=-k<0$ that for any  $w_0<w_0^{\sharp}$
and $x\in [0, q(w_0,\gamma_2(w_0))]$,
\begin{equation*}
\frac{1}{2}\sigma^2 w'(x; w_0,\gamma_2(w_0))-Mw(x; w_0,\gamma_2(w_0))\geq \gamma_2(w_0)-\pi(0)-h(x),
\end{equation*}
which yields that for $w_0<w_0^{\sharp}$,
\begin{align*}
&w(x_5^{\dag}; w_0,\gamma_2(w_0))\\
&\quad\geq w(x_4^{\dag}; w_0, \gamma_2(w_0))e^{\xi (x_5^{\dag}-x_4^{\dag})}
+\frac{2}{\sigma^2}\int_{x_4^{\dag}}^{x_5^{\dag}}\big[\gamma_2(w_0)-\pi(0)-h(y)\big]
e^{\xi(x_5^{\dag}-y)}
\, \mathrm{d}y\\
&\quad\geq\Big(\frac{\underline{f}}{x_4^{\dag}}-k\Big)e^{\xi (x_5^{\dag}-x_4^{\dag})}
+\frac{2}{\sigma^2}\int_{x_4^{\dag}}^{x_5^{\dag}}\big[\gamma_2(w_0)-\pi(0)-h(y)\big]
e^{\xi(x_5^{\dag}-y)}\, \mathrm{d}y,
\end{align*}
where $\xi=2M/\sigma^2$.
Therefore, \eqref{eq:lim-gamma2=infty} immediately implies
\[
\lim_{w_0\to-\infty} w(x_5^{\dag}; w_0,\gamma_2(w_0))=\infty.
\]
However, it follows from $q(w_0,\gamma_2(w_0))>x_5^{\dag}$ and Lemma~\ref{lem:w-property} \eqref{solution_x_3} that $w(x_5^{\dag}; w_0,\gamma_2(w_0))\leq w(q(w_0,\gamma_2(w_0)); w_0,\gamma_2(w_0))=-k$. This contradiction implies
the second part of  \eqref{eq:limf2}.
\end{proof}

\section*{Acknowledgments}
We thank Professor Hanqin Zhang at the National University of Singapore for discussions about this problem.

\bibliographystyle{ormsv080}
\bibliography{yao,cao}
\end{document}